\documentclass[11pt,reqno,oneside]{amsart}

\usepackage{graphicx}
\usepackage{amssymb}
\usepackage{epstopdf}

\usepackage[a4paper, total={6in, 9.1in}]{geometry}

\usepackage{amsmath,amsfonts,amsthm,mathrsfs,amssymb,cite}
\usepackage[usenames]{color}

\usepackage{graphics}
\usepackage{framed}

\usepackage{enumerate}
\usepackage{hyperref}
\usepackage{xcolor}
\hypersetup{
  colorlinks,
  linkcolor={blue!80!black},
  urlcolor={blue!80!black},
  citecolor={blue!80!black}
}
\usepackage[capitalize,noabbrev]{cleveref}

\numberwithin{equation}{section}

\newtheorem{Theorem}{Theorem}[section]
\newtheorem{Lemma}[Theorem]{Lemma}
\newtheorem{Example}[Theorem]{Example}
\newtheorem{Proposition}[Theorem]{Proposition}
\newtheorem{Definition}[Theorem]{Definition}
\newtheorem{Corollary}[Theorem]{Corollary}
   
\newtheorem{Remark}[Theorem]{Remark}
\numberwithin{equation}{section}

 \def\p{\partial} 
\def \Vh0{\stackrel{\circ}{V}_h} \def\to{\rightarrow}

  \def\f{\frac}  
   \def\eps{\varepsilon}

\def\p{\partial}

\newcommand{\lc}
{\mathrel{\raise2pt\hbox{${\mathop<\limits_{\raise1pt\hbox
{\mbox{$\sim$}}}}$}}}

\newcommand{\gc}
{\mathrel{\raise2pt\hbox{${\mathop>\limits_{\raise1pt\hbox{\mbox{$\sim$}}}}$}}}

\newcommand{\ec}
{\mathrel{\raise2pt\hbox{${\mathop=\limits_{\raise1pt\hbox{\mbox{$\sim$}}}}$}}}

\def\bb{\begin{equation}} \def\ee{\end{equation}}
\def\bb{\begin{definition}} \def\ee{\end{definition}}
\def\beqn{\begin{eqnarray}}  \def\eqn{\end{eqnarray}}

\def\beqnx{\begin{eqnarray*}} \def\eqnx{\end{eqnarray*}}

\def\bn{\begin{enumerate}} \def\en{\end{enumerate}}

\def\bd{\begin{description}} \def\ed{\end{description}}



\title[Quantum Integrable Systems and plasmon resonances]{Quantum Integrable Systems and Concentration of Plasmon Resonance}

\author{Habib Ammari}
\address{Department of Mathematics, ETH Z\"urich, R\"amistrasse 101, CH-8092, Switzerland}
\email{habib.ammari@math.ethz.ch}

\author{Yat Tin Chow}
\address{Department of Mathematics, University of California, Riverside, USA}
\email{yattinc@ucr.edu}

\author{Hongyu Liu}
\address{Department of Mathematics, City University of Hong Kong, Hong Kong SAR, China}
\email{hongyu.liuip@gmail.com, hongyliu@cityu.edu.hk}

\author{Mahesh Sunkula}
\address{Department of Mathematics, Purdue University, West Lafayette, USA}
\email{msunkula@purdue.edu}

\begin{document}
\maketitle

\begin{abstract}
We are concerned with the quantitative mathematical understanding of surface plasmon resonance (SPR) {\color{black} when $d \geq 3$.} SPR is the resonant oscillation of conducting electrons at the interface between negative and positive permittivity materials and forms the fundamental basis of many cutting-edge applications of metamaterials. It is recently found that the SPR concentrates due to curvature effect. In this paper, we derive sharper and more explicit characterisations of the SPR concentration at high-curvature places in both the static and quasi-static regimes. The study can be boiled down to analyzing the geometries of the so-called Neumann-Poincar\'e (NP) operators, which are certain pseudo-differential operators sitting on the interfacial boundary. We propose to study the joint Hamiltonian flow of an integral system given by the moment map defined by the NP operator. Via considering the Heisenberg picture and lifting the joint flow to a joint wave propagator, we establish a more general version of quantum ergodicity on each leaf of the foliation of this integrable system, which can then be used to establish the desired SPR concentration results. The mathematical framework developed in this paper leverages the Heisenberg picture of quantization and extends some results of quantum integrable system via generalising the concept of quantum ergodicity, which can be of independent interest to the spectral theory and the potential theory.

\medskip



\end{abstract}

\section{Introduction}\label{sect:1}

\subsection{Physical background and motivation}\label{sect:1.2}

In this paper, we are concerned with the quantitative mathematical understanding of surface plasmon resonance (SPR) {\color{black} when $d\geq 3$}. SPR is the resonant oscillation of conducting electrons at the interface between negative and positive permittivity materials and forms the fundamental basis of many cutting-edge applications of metamaterials. To motivate the study, we briefly discuss the mathematical setup of SPR. 

Let $D$ be a bounded $\mathcal{C}^{\infty}$ domain in $\mathbb{R}^d$, $d\geq 2$, with a connected complement $\mathbb{R}^d\backslash\overline{D}$. Let $\gamma_c$ and $\gamma_m$ be two real constants with $\gamma_m\in\mathbb{R}_+$ given and fixed. Let 
\begin{equation}\label{eq:mc1}
\gamma_{D}  = \gamma_c \chi(D) +  \gamma_m \chi(\mathbb{R}^d \backslash \overline{D}),
\end{equation}
where and also in what follows, $\chi$ stands for the characteristic function of a domain.
{\color{black} Consider the following homogeneous problem for a potential field $u\in H_{loc}^1(\mathbb{R}^d)$,
 \beqn
       \mathcal{L}_{\gamma_D} u = 0 \ \ \mbox{in}\ \ \mathbb{R}^d;\quad u(x) =  \mathcal{O}(|x|^{1-d}) \ \ \mbox{ as }\; |x| \rightarrow \infty,
    \label{transmission}
\eqn
where $\mathcal{L}_{\gamma_D} u:=\nabla(\gamma_D\nabla u)$. It is clear that $u\equiv 0$ is a trivial solution to \eqref{transmission}. If there exists a nontrivial solution $u$ to \eqref{transmission}, then $\gamma_c$ is called a plasmonic eigenvalue and $u$ is the associated plasmonic eigenfunction. } It is apparent that a plasmonic eigenvalue must be negative, since otherwise by the ellipticity of the partial differential operator (PDO) $\mathcal{L}_{\gamma_D} u:=\nabla(\gamma_D\nabla u)$, \eqref{transmission} admits only a trivial solution. That is, the negativity of $\gamma_c$ may enable that $\mathrm{Ker}(\mathcal{L}_{\gamma_D})\neq \emptyset$ which consists of the nontrivial solutions to \eqref{transmission}. In the physical scenario, the nontrivial kernel can induce a resonant field in a standard way. In fact, let us consider the following electrostatic problem for $u\in H_{loc}^1(\mathbb{R}^d)$:
\beqn
    \begin{cases}
        \nabla \cdot (\gamma_{D} \nabla u) = 0 \ \text{ in }\; \mathbb{R}^d, \\[1.5mm]
        (u-u_0)(x)=\mathcal{O}(|x|^{1-d})\ \mbox{as}\ |x|\rightarrow\infty,
    \end{cases}
    \label{transmissionE}
\eqn
where $u_0$ is a harmonic function in $\mathbb{R}^d$ that signifies an incident field, and $u$ is the incurred electric potential field. In the physical setting, $\gamma_c$ and $\gamma_m$ respectively specify the dielectric constants of the inclusion $D$ and the matrix space $\mathbb{R}^d\backslash\overline{D}$. 
If $\gamma_c$ is a plasmonic eigenvalue and moreover if $u_0$ is properly chosen in a way that $\mathcal{L}_{\gamma_D}u_0$ sits in the space spanned by $\mathcal{L}_{\gamma_D}$ acting on the plasmonic eigenfunctions, it is clear that a resonant field can be induced which is a linear superposition of the fields in $\mathrm{Ker}(\mathcal{L}_{\gamma_D})$. It is not surprising that the resonant field exhibits a highly oscillatory pattern. However, it is highly intriguing that the high oscillation mainly propagates along the material interface, namely $\partial D$. This peculiar phenomenon is referred to as the surface plasmon resonance (SPR). The SPR forms the fundamental basis for an array of frontier industrial and engineering applications including highly sensitive biological detectors to invisibility cloaks \cite{BS,DLZ3,FM,Kli,LZ,MN,OI,Sch,SPW,Z}. Its theoretical understanding also arouses growing interest in the mathematical literature \cite{ACKLM,Amm1,AJ,AKL,CGL,DLL3,G,LL3,shapiro,LL1,plasmon1,weyl2,weyl1}, especially its intriguing and delicate connection to the spectral theory of the Neumann-Poincar\'e (NP) operator as described in what follows. 

The NP operator is a classical weakly-singular boundary integral operator in potential theory \cite{book,kellog} and is defined by:
\beqn
    \mathcal{K}^*_{\partial D} [\phi] (x) := \f{1}{\varpi_d} \int_{\partial D} \f{\langle  x-y,\nu (x) \rangle  }{|x-y|^d} \phi(y) d \sigma(y), \quad x\in\partial D,
    \label{operatorK}
\eqn
where $\varpi_d$ signifies the surface area of the unit sphere in $\mathbb{R}^d$ and $\nu(x)$ signifies the unit outward normal at $x \in \partial D$. In studying the plasmonic eigenvalue problem \eqref{transmission}, we shall also need to introduce the following single-layer potential:
\beqn
   \mathcal{S}_{\partial D} [\phi] (x) := \int_{\partial D} \Gamma(x-y) \phi(y) d \sigma(y),\; x\in\mathbb{R}^d, \label{eq:sl1} 
\eqn
where $\Gamma$ is the fundamental solution to $-\Delta$ in $\mathbb{R}^d$ :
\beqn
    \Gamma (x-y) =
    \begin{cases}
     -\f{1}{2\pi} \log |x-y| & \text{ if }\; d = 2 \, ,\\
     \f{1}{(2-d)\varpi_d} |x-y|^{2-d} & \text{ if }\; d > 2 \, .
    \end{cases}
    \label{fundamental}
\eqn
The following jump relation holds across $\partial D$ for $\phi\in H^{-1/2}(\partial D)$: 
\beqn
    \f{\p}{\p \nu} \left(  \mathcal{S}_{\partial D} [\phi] \right)^{\pm}(x) = (\pm \f{1}{2} Id + \mathcal{K}^*_{ \partial D} )[\phi](x),\ \ x\in\partial D,
    \label{jump_condition}
\eqn
where $\pm$ signify the traces taken from the inside and outside of $D$ respectively, and $Id$ is the identity operator. {\color{black} Using \eqref{operatorK}--\eqref{jump_condition}, it can be directly verified that the plasmonic eigenvalue problem \eqref{transmission} is equivalent to the following spectral problem of determining $\lambda(\gamma_c,\gamma_m):=(\gamma_c+\gamma_m)/[2(\gamma_c-\gamma_m)]$ and a nontrivial surface density distribution $\phi\in H^{-1/2}(\partial D, d \sigma)$ such that:
\begin{equation}\label{eq:eigen1}
\begin{split}
u(x)=\mathcal{S}_{\partial D} [\phi](x), \ x\in \mathbb{R}^d;\quad
 \mathcal{K}_{\partial D}^*[\phi](x)=\lambda(\gamma_c, \gamma_m)\phi(x),\ x\in \partial D. 
\end{split}
\end{equation}
 That is, in order to determine the plasmonic eigenvalue $\gamma_c$ of \eqref{transmission}, it is sufficient to determine the eigenvalues of the NP operator $\mathcal{K}_{\partial D}^*$. } On the other hand, in order to understand the peculiar behaviour of the plasmonic resonant field, one needs to study the quantitative properties of the NP eigenfunctions in \eqref{eq:eigen1} as well as the associated single-layer potentials in \eqref{eq:sl1}. 

The NP operator $\mathcal{K}_{\partial D}^*$ is compact and hence its eigenvalues are discrete, infinite and accumulating at zero. A classical result is that $\lambda(\mathcal{K}_{\partial D}^*)\subset (-1/2, 1/2]$, which consolidates the negativity of a plasmonic eigenvalue in \eqref{eq:eigen1}. Due to their connection to the SPR discussed above, the quantitative properties of the NP eigenvalues have been extensively studied in recent years; see e.g. \cite{ACKLM2,LLL1,curv_Liu_3,KLY,weyl2,weyl1} and the references cited therein. As is mentioned earlier that the SPR mainly oscillates around the material interface $\partial D$, which is rigorously justified in \cite{AKMN}. It is found mainly through numerics in \cite{curv_Liu_3} that the SPR tends to concentrate at high-curvature places on $\partial D$. In \cite{ACL2}, a theoretical understanding is established to investigate such a peculiar curvature effect of the SPR when $D$ is convex. In \cite{DLZ,DLZ2}, a specific (possibly curved) nanorod geometry was considered, and it is shown that the SPR concentrates at the two ends of the nanorod, where both the mean and Gaussian curvatures are high. In this paper, by developing and exploring new technical tools, we shall derive sharper and more explicit characterisations of the SPR concentration phenomenon driven by the extrinsic curvature {\color{black} when $d \geq 3$.} It is remarked that according to the discussion in \cite{ACL2}, in order to study the SPR concentration, it is sufficient for us to consider the concentration of the NP eigenfunctions in \eqref{eq:eigen1} driven by the extrinsic curvature. Moreover, in addition to the static problem \eqref{transmission}, we shall also consider the study in the quasi-static regime, which shall be described in Section~6. Finally, we would like to mention in passing some related works on the polariton resonance associated with elastic metamaterials \cite{AKM,DLL,DLL2,LLL2,LL2,MR3} and the mathematical framework developed in this paper can be extended to study the geometric properties of the polartion resonance. 

\subsection{Discussion of the technical novelty}

In order to provide a global view of the technical contributions in this article, we briefly discuss the mathematical strategies and the new technical tools that are proposed and developed for tackling the concentration of the NP eigenfunctions and hence the SPR. 

The layer potential operators are pseudo-differential operators whose principal symbols encode the geometric characters of $\partial D$ {\color{black} when $d \geq 3$. It is pointed out that they are of Cauchy type when $d = 2$.} Our main idea in this work is to analyze the quantum ergodicity properties of these operators {\color{black}when $d \geq 3$ and under an addtional assumption, which we refer to as Assumption (A) in Section \ref{sect:3}.} To that end, we study the joint Hamiltonian flow of commuting Hamiltonians, one of which is the principal symbol of the NP operator. 
Using this, we derive a new version of generalized Weyl's law that the asymptotic average of the magnitude of the joint eigenfunctions of the joint spectrum sitting inside a given polytope in a neighbourhood of each point is directly proportional to a weighted volume of the pre-image of the polytope by the moment map at the respective point. {\color{black} In particular, we would like to point out that it generalizes the related results in \cite{PelayoVuNgoc11, VuNgoc06, VuNgoc07} for quantum integrable system, as well as \cite{Sogge} where a pointwise generalized Weyl's law of the Laplacian is proved.}

Then, we lift the joint Hamiltonian flow to a joint wave propagator via the Heisenberg picture. We obtain a quantum ergodicity result on each leaf of the foliation of the underlying integral system. {\color{black} This extends the classical results in the literature \cite{erg1,erg11, Toth02, Toth13, erg32,erg33,erg34,erg35, erg2,erg4}.} 
By using our established quantum ergodicity result, we further obtain a subsequence (of density one) of eigenfunctions such that their magnitude weakly converges to a weighted average of ergodic measures over each of the leaf, where this weighted average at different points again relates to the volume of the pre-image of the polytope by the moment map at the respective point. We provide explicit upper and lower bounds of the aforementioned volume as functions only depending on the principal curvatures. When the joint flow is ergodic with respect to the Liouville measures on each of the leaves, we obtain a more explicit description of the weighted average. With that, we provides a more explicit and sharper characterization of the localization of the plasmon resonance driven the associated extrinsic curvature at a specific boundary point {\color{black}when $d \geq 3$.} In fact, we provide an explicit and motivating example of a manifold with rotational symmetry, where the joint flow and the Lagrangian foliation can be explicitly worked out, and the bounds via the principal curvatures can also be calculated explicitly. From our result, we have associated the quantitative understanding of the plasmon resonances to the dynamical properties of the Hamiltonian flows. 
{\color{black} To our best knowledge, the first time when quantum integrable system is considered to show eigenfunction concentration on Lagrangian submanifolds is in \cite{Toth02}, where the Laplacian eigenfunctions are discussed.}
%

Finally, we would like remark that at the first glance, it is a bit paradoxical to still hold the name of quantum ergodicity when we are investigating a quantum integrable system, since as is conventionally known that the descriptions of a (complete) integrable system and that of ergodicity are almost on the opposite sides of the spectrum of a dynamical system. However, our discussion is on the ergodicity on the leaves of the foliation given by the integrable system, say e.g. the Lagrangian tori if we have a complete integrable system, and therefore no paradox emerges.

The rest of the paper is organized as follows. Sections 2 and 3 are devoted to preliminaries for the sake of completeness and self-containedness of the paper. In Section 2, we briefly recall the principal symbols of the layer-potential operators following the discussions in \cite{ACL,ACL2} 
{\color{black} as well as \cite{Taylor,ola}.} In Section 3, we provide a general and brief introduction to quantum integrable systems. In Section 4, we establish a generalized Weyl's law over quantum integrable systems for our purpose, and generalize the argument of the quantum ergodicity over each leaf of the folliation to obtain a variance-like estimate.
Section 5 and 6 are respectively devoted to the quantitative results of the concentration of the plasmon resonances in the static and quasi-static regimes {\color{black} when $d \geq 3$.}


\section{Potential operators as pseudo-differential operators}\label{sect:2} 

\subsection{${h}$-pseudodifferential operators}\label{sec:3.3}
Let us consider the manifold $M = \mathbb{R}^{2d}$ or $M = T^{*}X$, with the symplectic form $\displaystyle{\omega = \sum_{i = 1}^{d}dx_i \wedge d\xi_i}$, where $X$ is a $d-$dimensional closed manifold. The 
${h}-$pseudo differential operators acting on the Hilbert space $\mathcal{H} = L^2(\mathbb{R}^d)$ (or $L^2(X)$) give semiclassical operators. To start with, we let ${{\mathcal{S}}}^m(\mathbb{R}^{2d})$ be the H\"ormander class (symbol class) of order $m$ whose elements are functions $f$ in the
space $C^{\infty}(\mathbb{R}^{2d})$ such that, for $m \in \mathbb{R}$, 
\begin{equation}\label{eq:class}
|\partial^{\alpha}_{(x, \xi)}f| \leq C_{\alpha}\langle(x, \xi)\rangle^{m},\ \ (x,\xi)\in\mathbb{R}^{2d},
\end{equation}
for every $\alpha \in \mathbb{N}^{2d}$. Here, $\langle z \rangle := (1+|z|^2)^{\frac{1}{2}}$. 
\begin{Definition}
    Let $f \in {{\mathcal{S}}}^m(\mathbb{R}^{2d})$. The  \emph{${h}-$pseudodifferential operators}  of $f$ are given on the Schwartz space ${\tilde{\mathcal{S}}}(\mathbb{R}^d)$ by the expressions: 
    \begin{eqnarray*}
 \text{(Left)} \quad  (\mathrm{Op}^L_{f,{h}}(u))(x)&:=& \frac{1}{(2\pi{h})^d}\int_{\mathbb{R}^d}\int_{\mathbb{R}^d}\exp\left(\frac{\mathrm{i}}{{h}}(x-y)\cdot \xi\right) f\left(y, \xi\right) u(y) \, dy \, d\xi; \\
    \text{(Weyl)} \quad  (\mathrm{Op}^W_{f,{h}}(u))(x)&:=& \frac{1}{(2\pi{h})^d}\int_{\mathbb{R}^d}\int_{\mathbb{R}^d}\exp\left(\frac{\mathrm{i}}{{h}}(x-y)\cdot \xi \right) f\left(\frac{x+y}{2}, \xi\right) u(y) \, dy \, d\xi; \\
 \text{(Right)} \quad  (\mathrm{Op}^R_{f.{h}}(u))(x)&:=& \frac{1}{(2\pi{h})^d}\int_{\mathbb{R}^d}\int_{\mathbb{R}^d}\exp\left(\frac{\mathrm{i}}{{h}}(x-y)\cdot \xi\right) f\left(x, \xi\right) u(y) \, dy \, d\xi.
    \end{eqnarray*}
\end{Definition}
 \emph{${h}-$pseudodifferential operators} also give rise to semiclassical operators on $M = T^*X$,($X$ is a closed $d-$dimentional manifold). Let $X$ be covered by a collection of smooth charts $\left\{U_1, \cdots, U_\ell\right\}$, such that each $U_i$, $1\leq i\leq \ell$, is a convex bounded domain of $\mathbb{R}^d$. There exists a partition of unity $\chi_1^2, \cdots, \chi_\ell^2$ which is subordinate to the cover $\left\{U_1, \cdots, U_\ell\right\}$. Let ${\mathcal{S}}^m(T^{*}X)$ be the space of functions $f$ in the
 space $C^{\infty}(T^*X)$ such that, for $m \in \mathbb{R}$, 
 \begin{equation}
 |\partial^{\alpha}_x\partial^{\beta}_{\xi}f(x, \xi)| \leq C_{\alpha, \beta}\langle(\xi)\rangle^{m-|\beta|},
 \end{equation}
 for every $\alpha, \beta \in \mathbb{N}^{n}$. Define the operator on $X$ to be 
 \begin{eqnarray}\label{eq:pseudoX}
\mathrm{Op}^{L/W/R}_{f,{h}}(u):= \sum_{j=1}^{\ell} \chi_j\cdot {(\mathrm{Op}^{L/W/R} )}_{f,{h}}^j(\chi_ju), \ \ u \in C^{\infty}(X), \end{eqnarray}
where ${(\mathrm{Op}^{L/W/R})}_{{h}}^j(f)$ are the pseudodifferential operators on $U_j$ with the principal symbol $f\chi_j^2$.  Following \cite{PPV14}, we have that the operators $\mathrm{Op}^{L/W/R}_{{h}}(f)$ are all pseudodifferential operator on $X$ with the principal symbol $f$.
\begin{Proposition}
    Let ${\mathcal{S}}^m(T^{*}X)$ be a H\"ormander class, $\mathbf{I} = (0, 1], h \in \mathbf{I}$,
    and $\mathcal{H}_{{h}}  = L^2(X)$ (independent of ${h}$). Then the pseudodifferential operators on $X$, namely all of the above quantizations $\mathrm{Op}^{L/R}_{h}(f)$ and $\mathrm{Op}^W_{h}(f)$ defined in \eqref{eq:pseudoX}, form a space of semiclassical operators. 
\end{Proposition}
\begin{proof} We refer the readers to \cite{PPV14} for a proof of this theorem for $\mathrm{Op}^W_{f,h}$.
It is noted that after applying the operator $\exp\left(\pm i \frac{h}{2} \partial_x \partial_\xi\right)$, the Weyl quantization $\text{Op}^{W}_{f,h}$ and left/right quantizations $\text{Op}^{L/R}_{f,h}$ differ only in the higher order term. Hence, one can conclude that the Beal's criterion applies to $\mathrm{Op}^W_{f,h}$ if and only if it applies to $\text{Op}^{L/R}_{f,h}$, which readily completes the proof. 
    \end{proof}

From now on, whenever we do not specify whether it is left, right or Weyl, we presume  $\text{Op}_{f,h} := \text{Op}^R_{f,h}$  is the right quantization.   We notice that Weyl qanitization is symmetric in the $L^2$ metric by definition. In fact, if we do not specify the cover $\{U_i \}_{1\leq i \leq l}$, an operator so defined (via any of the quantization $\mathrm{Op}^{L/W/R}_{f,{h}}$) is unique up to $h \Phi \text{SO}^{m-1}_h$ if $f \in \mathcal{S}^m(T^*(\partial D))$) belonging to the symbol class of order $m$.

\subsection{Geometric description of $\partial D$}
For the subsequent need, we briefly introduce the geometric description of $D\subset\mathbb{R}^d$. Let $\mathbb{X}: \mathbf{u} = (u_1, u_2, ..., u_{d-1})\in U \subset \mathbb{R}^{d-1} \rightarrow \mathbb{X}(u)\in\partial D \subset \mathbb{R}^{d}$ be a regular parametrization of the surface $\partial D$ and let $\mathbb{X}_j := \f{\p \mathbb{X}}{\p u_j}$, $j=1,2,\ldots, d-1$.
We denote $\times_{j=1}^{d-1} \mathbb{X}_j = \mathbb{X}_1 \times \mathbb{X}_2 ... \times \mathbb{X}_{d-1} $. Since $\mathbb{X}$ is regular, we know $\times_{j=1}^{d-1} \mathbb{X}_j$ is
non-zero, and the normal vector $\nu := \times_{j=1}^{d-1} \mathbb{X}_j / |\times_{j=1}^{d-1} \mathbb{X}_j | $
is well-defined. Let $\bar{\nabla}$ be the standard covariant derivative on the ambient space $\mathbb{R}^{d}$, and $\textbf{II}$ be the second fundamental form given by
\beqn
\textbf{II}(\mathbf{v},\mathbf{w})= - \langle \bar{\nabla}_{\mathbf{v}} \nu, \mathbf{w} \rangle \nu =  \langle \nu,  \bar{\nabla}_{\mathbf{v}}  \mathbf{w} \rangle \nu,\ (\mathbf{v}, \mathbf{w})\in T(\p D) \times T(\p D). \notag
\eqn
Define
\beqn
\mathscr{A}(x) := ( \mathscr{A}_{ij}(x) ) =  \langle \textbf{II}_x (\mathbb{X}_i ,\mathbb{X}_j), \nu_x \rangle \notag \,,\ \ x \in \partial D.
\eqn
Let $g=(g_{ij})$ be the induced metric tensor on $\partial D$ and $(g^{ij}) = g^{-1}$. 
Finally, we write $\mathscr{H}(x), x \in \partial D$ as the mean curvature satisfying 
$$ \text{tr}_{g(x)} (\mathscr{A}(x)) :=  \sum_{i,j = 1}^{d-1} g^{ij}(x) \mathscr{A}_{ij}(x) := (d-1) \mathscr{H} (x). $$ 
Throughout the rest of the paper, we always assume $\mathscr{A}(x) \neq 0$ for all $x \in \partial D$.

\subsection{Principal symbols of layer potential operators}

Throughout the rest of the paper, with a bit abuse of notations, we shall also denote by $\mathcal{S}_{\partial D}$ the single-layer potential operator which is given in \eqref{eq:sl1} but with $x\in\partial D$. This should be clear from the context in what follows. We let $\mathcal{K}_{\partial D}$ signify the $L^2(\partial D, d \sigma)$-adjoint of the NP operator $\mathcal{K}^*_{\partial D}$. $\mathcal{K}_{\partial D}^*$ is symmetrizable on $H^{-1/2}(\partial D , d \sigma)$ (cf., e.g., \cite{putinar}) due to the following Kelley symmetrization identity:
\begin{equation}\label{eq:s3}
\mathcal{S}_{\partial D}  \, \mathcal{K}^*_{\partial D} =   \mathcal{K}_{\partial D} \,  \mathcal{S}_{\partial D}.
\end{equation}

In this section, we treat the layer potential operators as pseudodifferential operators {\color{black} when $d \geq 3$} and derive several important properties, especially their principal symbols. In fact, the special three-dimensional case was treated in \cite{weyl1,weyl2}, whereas the general case was considered in \cite{ACL,ACL2} {\color{black} as well as 
Chapter 12, Section C, Proposition C1 in \cite{Taylor} and Proposition 2.2 in \cite {ola}.} Since this result forms the starting point for our subsequent analysis, we discuss the main ingredients as follows. Before that, we introduce a slightly more relaxed symbol class $\tilde{\mathcal{S}}^m(T^*(\partial D))$ (compared to ${\mathcal{S}}^m(T^*(\partial D))$):
\[
\begin{split}
 \bigcup_{i} U_i = \partial D \,,\quad & F_i : \pi^{-1} (U_i) \rightarrow U_i \times \mathbb{R}^{d-1}\,,  \quad  \sum_i \psi_i^2 = 1 \,, \quad \text{supp}(\psi_i) \subset U_i\, ; \\
 \widetilde{S}^m (U_i \times \mathbb{R}^{d-1}\backslash\{0\} )  := & \bigg\{ a:  U_i \times ( \mathbb{R}^{d-1} \backslash \{0\} )  \rightarrow \mathbb{C} \, ;\\
&\qquad a \in \mathcal{C}^{\infty} (  U_i \times ( \mathbb{R}^{d-1} \backslash \{0\} )  ) \, , \, | \partial_\xi^\alpha \partial_x^{\beta} a (x,\xi) | \leq C_{\alpha, \beta} ( |\xi| )^{ m - |\alpha| } \bigg\};\\
\widetilde{S}^m (T^*(\partial D) )  :=&  \bigg \{ a: T^*(\partial D)  \backslash \partial D \times \{0\}  \rightarrow \mathbb{C} \, ; \\
 &\qquad \, a = \sum_{i} \psi_i F_i^* \left(   [F_i^{-1}]^*(\psi_i) \, a_i  \right) , a_i \in \widetilde{S}^m (U_i \times \mathbb{R}^{d-1} \backslash \{0\} ) \bigg \};
\end{split}
\]
where $\pi: T^*(\partial D) \rightarrow \partial D$ is the bundle projection. Similar to our discussion in Section~\ref{sec:3.3},
for a symbol $a \in \tilde{\mathcal{S}}^m(T^*(\partial D))$, we can define $\mathrm{Op}_{a,h}$ to be the $h$-pseudodifferential operator. In the sequel, we let $\widetilde{\Phi \text{SO}}_h^{m}$ denote the class of pseudodifferential operators of order $m$ associated with $\tilde{\mathcal{S}}^m(T^*(\partial D))$. We also let $\widetilde{\Phi \text{SO}}^{m} := \widetilde{\Phi \text{SO}}_1^{m}$, namely $h=1$.

\begin{Theorem}
Assume that $\partial D \in C^{\infty}$. {\color{black} When $d \geq 3$,} the operators $\mathcal{K}^*_{\partial D}$ and $\mathcal{S}_{\partial D}$ are pseudodifferential operators of order $-1$ with their symbols given as follows in the geodesic normal coordinate around each point $x$:
\begin{equation}\label{eq:s1}
\begin{split}
p_{\mathcal{K}^*_{\partial D}}(x,\xi) =& p_{\mathcal{K}^*_{\partial D} , -1}(x,\xi)  +  \mathcal{O}(|\xi|^{-2})\\
   =&  (d-1) \mathscr{H}(x) \,  |\xi|^{-1} -  \langle \mathscr{A}(x)  \, \xi, \, \xi \rangle \, |\xi|^{-3}
+  \mathcal{O}(|\xi|^{-2}) \, ,
\end{split}
\end{equation}
and
\begin{equation}\label{eq:s2}
\begin{split}
p_{\mathcal{S}_{\partial D}}(x,\xi) = p_{\mathcal{S}_{\partial D},-1}(x,\xi) +  \mathcal{O}(|\xi|_{g(x)}^{-2}) =  \frac{1}{2} |\xi|_{g(x)}^{-1} +  \mathcal{O}(|\xi|_{g(x)}^{-2}) \,,
\end{split}
\end{equation}
where the asymptotics $\mathcal{O}$ depends on $\| \mathbb{X} \|_{\mathcal{C}^2} $. The result in \eqref{eq:s1} holds also for $\mathcal{K}_{\partial D}$ if only the leading-order term is concerned. 

Using the symmetrization identity \eqref{eq:s3} and the self-adjointness of $\mathcal{S}_{\partial D}$, we have
\begin{equation}\label{eq:s4}
\begin{split}
\mathcal{K}^*_{\partial D} =& |\mathcal{D}|^{-1} \Bigg \{ (d-1) \mathscr{H}(x) \Delta_{\partial D}\\
&\hspace*{1cm} - \sum_{i,j,k,l = 1}^{d-1} \frac{1}{ \sqrt{|g(x)|} } \partial_i g^{ij}(x) \sqrt{|g(x)|} \mathscr{A}_{jk}(x) g^{kl}(x) \partial_{l} \Bigg\} |\mathcal{D}|^{-2} \text{ mod } \widetilde{\Phi\mathrm{SO}}^{-2},\medskip \\
\mathcal{S}_{\partial D} =& \frac{1}{2} |\mathcal{D}|^{-1}  \text{ mod } \widetilde{\Phi\mathrm{SO}}^{-2},
\end{split}
\end{equation}
where $ \Delta_{\partial D}$ is the surface Laplacian of $\partial D$, and $|\mathcal{D}|^{-1} := \mathrm{Op}_{|\xi|_{g(x)}^{-1}} $. Moreover, we have $\mathcal{K}^*_{h,\partial D} := \frac{1}{h} |\mathcal{D}|^{-\frac{1}{2}}   \mathcal{K}^*_{\partial D} |\mathcal{D}|^{\frac{1}{2}}$, which 
is self-adjoint up to $\text{mod } h \widetilde{\Phi\mathrm{SO}}^{-2}_h$. 
\end{Theorem}

Finally, we note that $\left(\tilde{\lambda^i}^2, \phi^i\right)$ is an eigenpair of $ \mathcal{K}^*_{\partial D} $ if and only if
$\left(\frac{\tilde{\lambda^i}}{h},  |D|^{-\frac{1}{2}}  \phi^i\right)$ is an eigenpair of $  \mathcal{K}^*_{h,\partial D}  $.
Hencefore, we write
\begin{equation}\label{eq:eg1}
 ( \tilde{\lambda^i}^2(h),  \phi^i(h) ) := \left(\frac{\tilde{\lambda^i}^2}{h},  |D|^{-\frac{1}{2}}  \phi^i\right).
 \end{equation}

\section{Classical and Quantum Integrable Systems} 
In this section, we give a brief review of the classical and quantum integrable systems, which shall be needed in our subsequent analysis.
\subsection{Classical integrable systems}
Let $M$ be a $2d-$dimensional symplectic manifold with a non-degenerate $2-$form $\omega$.  
\begin{Definition}\label{def:comp-int} 
    A \emph{completely integrable Hamiltonian system} $(M, \omega, F)$ on a $2d$-di\-men\-sio\-nal symplectic manifold $(M, \omega)$ is given by a set of $d$ smooth functions $H_1, \ldots, H_d \in C^\infty(M)$, that are functionally independent and Poisson-commuting, i.e., 
    \[
    \{H_i, H_j\}  := - \omega(X_{H_i}, X_{H_j}) = 0, \qquad i,j \in \{1, \ldots, d\} ,
    \]
where we recall that $X_{H_i}$ is the symplectic gradient vector field given by
\beqnx
\iota_{X_{H_i}} \, \omega =  d H_i \,.
\eqnx
    The map $F=(H_1,\ldots,H_d): M \to \mathbb{R}^d$ is called the {\em moment map}.  
\end{Definition} 

The level sets of the moment map in a completely integrable system form a Lagrangian foliation $F: M \rightarrow \mathbb{R}^d$.

\begin{Definition}\label{def:reg-sing-pt}
    Let $F= (H_1,\ldots,H_d)$ be the moment map of a completely integrable system on $\mathbb{R}^{2d}$.  
    A point $\mathfrak{m} \in \mathbb{R}^{2d}$ is said to be a \emph{regular point} if  
    $$\mathrm{rank}\{X_{H_1}(\mathfrak{m}), \ldots, X_{H_d}(\mathfrak{m})\} = d \ . $$ 
    If  
    \[
    \mathrm{rank}\{X_{H_1}(\mathfrak{m}), \ldots, X_{H_d}(\mathfrak{m})\} = r, \qquad 0 \leq r < d \ , 
    \] 
    then the point $\mathfrak{m} \in \mathbb{R}^{2d}$ is said to be a \emph{singular point of rank $r$}. The value $F(\mathfrak{m}) \in \mathbb{R}^d$ is called a \emph{regular value} if $\mathfrak{m}$ is a regular point and a \emph{singular value} if $\mathfrak{m}$ is a singular point. 
\end{Definition}

Suppose that  $\mathfrak{m} \in \mathbb{R}^{2d}$ is a singular point of rank $r$  for a completely integrable system $F= (H_1,\ldots,H_d)$ on $\mathbb{R}^{2d}$. After replacing the $\displaystyle{H_{i}\text{'s}}$ with invertible linear combinations of $H_j$'s if necessary, we may assume that 
$$\displaystyle{X_{H_{1}}(\mathfrak{m}) = \cdots = X_{H_{d-r}}(\mathfrak{m}) = 0},$$ and the $X_{H_i}$'s are linearly independent for $d-r < i \leq d$. 
The quadratic parts of $H_1, \ldots, H_{d-r}$ form an abelian subalgebra $\mathfrak{s}_\mathfrak{m}$ of the Lie algebra of quadratic forms, with the Poisson bracket as the Lie bracket.  

\begin{Definition}    \label{def:non-deg-pt}A singular point $\mathfrak{m}$ or rank $r$ is said to be a \emph{non-degenerate singular point of rank $r$} if the sub-algebra $\mathfrak{s}_\mathfrak{m}$ is a Cartan sub-algebra  of the Lie algebra $\mathfrak{sp}(2d-2r, \mathbb{R})$ of the symplectic group $\mathrm{Sp}(2d-2r, \mathbb{R})$. 
\end{Definition}

\begin{Remark} 
    In an obvious way, Definitions \ref{def:reg-sing-pt} and \ref{def:non-deg-pt} can be carried over to a completely integrable system $(M, \omega, F)$  on a general $2d$-dimensional symplectic manifold. 
    
\end{Remark}

In 1936, Williamson \cite{Williamson} classified the Cartan subalgebras of the Lie algebra of the symplectic group. 

\begin{Theorem}[{\bf Williamson}]\label{thm:williamson}
    Let $\mathfrak{s} \subset \mathfrak{sp}(2l; \mathbb{R})$ be a Cartan subalgebra. Then there exist canonical coordinates $(q_1, \ldots, q_l, p_1, \ldots, p_l)$ for $\mathbb{R}^{2l}$, a triple $(k_{el}, k_{hy}, k_{ff}) \in \mathbb{Z}^3_{\geq 0}$ satisfying the condition $k_{el}+k_{hy}+2 k_{ff} = l$, and a basis $f_1, \ldots,f_l$ of $\mathfrak{s}$ such that 
    \begin{align*}
    f_i &= \frac{q_i^2+p_i^2}{2},  && \ \, i = 1, \ldots, k_{el} , \hspace*{20mm} \\[2mm]
    f_j &= q_jp_j,  && \ \, j = k_{ell}+1, \ldots, k_{el}+k_{hy}, \\[2mm]
    f_k &= 
    \left\{
    \begin{array}{l}
    \!q_kp_k+q_{k+1}p_{k+1},  \\[2mm]
    \!q_kp_{k+1}-q_{k+1}p_k, 
    \!\end{array}
    \right. 
    &&
    \begin{array}{l}
    k = k_{el}+k_{hy}+1, k_{el}+k_{hy}+3, \ldots, l-1 , \\[2mm]
    k = k_{el}+k_{hy}+2, k_{el}+k_{hy}+4, \ldots, l .
    \!\end{array}
    \end{align*}
    Additionally, two Cartan subalgebras $\mathfrak{s}, \mathfrak{s}^{'} \subset \mathfrak{sp}(2l; \mathbb{R})$ 
    are conjugate if and only if their corresponding triples are equal. 
    
    The elements of the basis of $\mathfrak{s}$ are called \emph{elliptic blocks}, \emph{hyperbolic blocks} or \emph{focus-focus blocks} according to whether they are of the form $\frac{q_i^2+p_i^2}{2}$, $q_jp_j$ or a pair  $q_kp_k+q_{k+1}p_{k+1}, q_kp_{k+1}-q_{k+1}p_k$, respectively. 
\end{Theorem}

%
Given a completely integrable system $\bigl(M, \omega, F = (H_1,\ldots, H_d)\bigr)$. Suppose $\mathfrak{m} \in M$ is a non-degenerate singularity of rank $r$. Then with the help of Williamson's Theorem, locally one can write the Hamiltons $H_i$ as $f_i$ for $i = 1, \cdots, k_{el}+k_{hy}+2k_{ff}$, and $H_i = p_i$ for $i = k_{el}+k_{hy}+2k_{ff}+1, \cdots, k_{el}+k_{hy}+2k_{ff}+r = n$.  

\subsection{Quantum Integrable Systems}
Next, we would like to provide a tool for the discussion of the lift of the classical Hamiltonian system to its operator counterpart.   For this purpose, we define the quantum integrable system.

Let $M$ be a $2d-$dimensional symplectic manifold with a non-degenerate $2-$form $\omega$. 
Let $\mathbf{I} \subset (0, 1]$ be any set that accumulates at $0$. If $\mathcal{H}$ is a complex Hilbert space, we denote by $\mathcal{L(H)}$ the set of linear (possibly unbounded) self-adjoint operators on $\mathcal{H}$ with a dense domain. 
\begin{Definition}
 A space $\Psi$ of \emph{semiclassical operators} is a subspace of $\displaystyle{\prod_{{h} \in \mathbf{I}} \mathcal{L}(\mathcal{H}_{{h}})}$, containing the identity, and equipped with a weak principal symbol map, which is an $\mathbb{R}$-linear map
 \begin{equation}\label{eq:sc-op}
 \sigma: \Psi \mapsto \mathcal{C}^{\infty}(M; \mathbb{R}),
 \end{equation}
 with the following properties:
 \begin{enumerate}
     \item $\sigma(Id) = 1$; (normalization)
     \item if $P, Q \in \Psi$ and if $P \circ Q$ is well defined and is in $\Psi$, then $\sigma(P\circ Q) = \sigma(P)\sigma(Q)$; (product formula) 
     \item if $\sigma(P) \geq 0$, then there exists a function ${h} \mapsto \epsilon({h})$ tending to zero as ${h} \rightarrow 0,$ such that $P \geq -\epsilon({h})$, for all ${h} \in \mathbf{I}$. (wear positivity)
 \end{enumerate} 
If $\displaystyle{P = (P_{{h}})_{{h}\in \mathbf{I}}}$, then $\sigma(P)$ is called the \emph{principal symbol} of P.
\end{Definition}
Such a family of Hilbert spaces can be obtained e.g. by the Weyl quantization (which we will specify later) or the geometric quantization with
complex polarizations.
\begin{Definition}
    A \emph{quantum integrable system} on $M$ consists of $d$ semiclassical operators
  \[ P_1 =  (P_{1,h}), \cdots, P_d = (P_{d,h})\] 
    acting on $\mathcal{H}_{{h}}$ which commute, i.e., $[P_{i,h} , P_{j,h}] = 0$ for all $i, j \in \left\{1, 2, \cdots, d\right\}$, for all ${h}$ and whose principal symbols $f_1 := \sigma (P_{1}), \cdots, f_d := \sigma (P_{d})$ form a completely integrable system on $M$.
\end{Definition}

\begin{Definition}
  Suppose $P$ and $Q$ are commuting semiclassical operators on $\mathcal{H}_{{h}}$. Then the \emph{joint spectrum} of $(P_{{h}}, Q_{{h}})$ is the support of the joint spectral measure, which is denoted as $\Sigma(P_{{h}}, Q_{{h}})$. If $\mathcal{H}_{{h}}$ is a finite dimensional (or, more generally, when the joint spectrum is discrete), then 
  \begin{equation}{\label{eq:JS}}
  \Sigma(P_{{h}}, Q_{{h}}) = \left\{(\lambda_1, \lambda_2) \in \mathbb{R}^2 : \exists v \neq 0, P_{{h}}v = \lambda_1v, Q_{{h}}v = \lambda_2v \right\}.
  \end{equation}
  The joint spectrum of $P, Q$, denoted by $\Sigma(P,Q)$, is the collection of all joint spectra of $(P_{{h}}, Q_{{h}})$, ${h} \in \mathbf{I}.$
\end{Definition}
Suppose that $(P_{1,h}), \cdots, (P_{d,h})$ form a quantum integrable system on $M$. Then the \emph{joint spectrum of} $\Big((P_{1,h}), \cdots, (P_{d,h})\Big)$ is  
\begin{equation}\label{eq:jspect}
\displaystyle{\Sigma\Big((P_{1,h}), \cdots, (P_{d,h})\Big) := \left\{(\lambda_1, \cdots, \lambda_d) \in \mathbb{R}^d: \bigcap_{i = 1}^{d}\ker(P_{i,h}-\lambda_{i}I) \neq \{0\}\right\}}.
\end{equation}

We would like to remark that in case the operators $P_{i,h}$ are not bounded, the commuting property of the operators 
is understood in the strong sense: the spectral measures (obtained via the spectral
theorem as a projector-valued measure) of $P_{i,h}$ and $P_{j,h}$ commute.

\section{Generalized Weyl's law }\label{sect:3}

In this section, we recall the concept of generalized Weyl's law and quantum ergodicity from the pioneering works of Shnirelman \cite{erg1,erg11}, Zelditch \cite{erg32,erg33,erg34,erg35}, Colin de Verdiere \cite{erg2} and Helffer-Martinez-Robert \cite{erg4},  {\color{black} as well as \cite{Sogge} where a pointwise generalized Weyl's law of the Laplacian is proved,} and generalize them to the case when we have a quantum integrable system \cite{PelayoVuNgoc11, VuNgoc06, VuNgoc07} for our purpose.
 For our subsequent use, we would like to further generalize it to provide a more reinforced description for the quantum integrable system.

\subsection{Hamiltonian flows of principal symbols.}
Consider the Hamiltonian $H: T^*(\partial D)  \rightarrow \mathbb{R}$:
\begin{equation}\label{eq:hf1}
        H(x, \xi) := [ p_{\mathcal{K}^*_{\partial D},-1}(x,\xi)  ]^2 \geq 0 \,.
\end{equation}
Throughout the rest of the paper, we impose the following assumption in our study.

\vskip 2mm

\noindent \textbf{Assumption (A)}  We assume $\langle \mathscr{A}(x) \, g^{-1}(x) \, \omega \, ,\,  g^{-1}(x) \, \omega \rangle \neq (d-1) \mathscr{H}(x)$ for all $x \in \partial D$ and $\omega \in \{ \xi :  |\xi |_{g(x)}^2 = 1 \} \subset T_x^*(\partial D)$.

\vskip 2mm
\noindent Assumption (A) holds if and only if $\overline{ \{H = 1\} } \bigcap \left(  \partial D  \times \{0\} \right)  = \emptyset$, which is further equivalent to the condition that the Hamiltonian $H \neq 0$ everywhere {\color{black} and hence the ellipticity of $\mathcal{K}^*_{\partial D}$.  As explored in Corollary 2.3 in \cite{ola}, it is clear that strictly convexity of $D$ implies Assumption (A).  Meanwhile, as discussed in \cite{ACL},  (at least) when $d=3$, with a quick application of the Gauss-Bonnet theorem, it yields that Assumption (A) holds if and only if $D$ is strictly convex.}
With this, gazing at \eqref{eq:eigen1}, it can be directly inferred that $\phi \in \mathcal{C}^{\infty} (\partial D)$.
 {\color{black} In this article, we always assume the validity of Assumption (A).}

Next, we set $\rho(r) = 1-\exp(-r): \mathbb{R}_+\rightarrow\mathbb{R}$. It is realized that $\rho(r)\geq 0$ and $\rho'(r) > 0$ for all $\mathbb{R}_+$. Moreover, $\rho(1/r^2) \in \mathcal{C}^{\infty} (\mathbb{R})$, with $\partial_r^{\ell}|_{r = 0} \left[ \rho(1/r^2) \right] = 0$ for all $\ell \in \mathbb{N}$ and 
\[
| \partial_r^{\ell} \rho(1/r^2) | \leq C_\ell (1+ |r|^2)^{\frac{-2 - \ell}{2}} \,.
\]
Define $\tilde{H}(x, \xi) =\rho( H(x, \xi)  ): T^*(\partial D)  \rightarrow \mathbb{R}$. {\color{black} When $d \geq 3$,} it can be directly verified that under Assumption (A) and together with the fact that $H \in \tilde{\mathcal{S}}^{-2} (T^*(\partial D))$, one has $\tilde{H} \in \mathcal{S}^{-2} (T^*(\partial D))$.

Let us now consider $k (\leq d)$ Poisson commuting and functionally independent Hamiltonians $f_1 = \tilde{H}, f_2, \cdots f_k \in {\mathcal{S}^{m}(T^{*}(\partial D))}$ for $m \geq -2$ (if $k = d$, then we will have a completely integrable system on $M$). Let $F = (f_1, \cdots, f_k)$ be the $k$-tuple of the above Hamiltonian functions. We also consider the corresponding  ${h}$-pseudodifferntial operators $\mathrm{Op}_{f_1,{h}}, \cdots, \mathrm{Op}_{f_k,{h}}$ acting on $\mathcal{H}_{{h}} = L^2(\partial D)$.

Next we consider the following solution under the (joint) Hamiltonian flows:
\begin{equation}\label{eq:hf2}
\begin{cases}
\frac{\partial}{\partial t_j} a(t_1,...,t_k) &=  \{ f_j , a (t_1,...,t_k)  \}, \medskip\\
a_{0} (x,\xi) &\in \mathcal{S}^m(T^{*}X),
\end{cases}
\end{equation}
which exists since $\{f_i,f_j\} = 0$, 
where we recall
$ \{ \cdot , \cdot \}$ is the Poisson bracket given by 
\beqnx
\{ f , g\} := X_f \, g = - \omega(X_f, X_g)\,.
\eqnx
With this notion in hand, we have $\frac{\partial}{\partial t_j} a = X_{f_j } a $, and it is clear that, writing $t = (t_1,...,t_k)$, we have
$ a(t) = a_{0} ( \gamma(t), p (t)) $
where 
\beqn
\begin{cases}
    \frac{\partial}{\partial t_j}  (\gamma(t), p(t) ) &= X_{f_j} (\gamma(t), p(t) ),\medskip  \\
    (\gamma(0), p(0) ) & = (x, \xi ) \in M.
\end{cases}
\label{ODE_joint_flow}
\eqn
To emphasize the dependence of $a$ on the initial value $(x,\xi)$, we also sometimes write
\beqnx
a_{(x,\xi)}(t) = a(t) \quad  \text{ with } \quad (\gamma(0), p(0) )  = (x, \xi ) \,.
\eqnx
Next we introduce the Heisenberg's picture and lift the above flow to the operator level via the well-known Egorov's theorem together with the commutativity of $\mathrm{Op}_{f_i, {h}}$ and $\mathrm{Op}_{f_j, {h}}$, for our situation.  Since this is a handy extension of the original Egorov's theorem (cf. \cite{Hor1,Hor2,Ego,ACL}), we only provide a sketch of the proof.

\begin{Proposition} \label{prop:4.1}
  Under Assumption (A), {\color{black} when $d \geq 3$,} we consider the following operator evolution equation for each $j \in \{1, \cdots, k\}$:
    \begin{equation}\label{eq:ohf1-QIS}
    \begin{cases}
    \frac{\partial}{\partial t_j} A_{{h}}(t) = \frac{\mathrm{i}}{ {h} } \left[ \mathrm{Op}_{f_j, {h}} , A_{{h}}(t) \right],\medskip\\
    A_{{h}}(0) = \mathrm{Op}_{a_{0}, {h}}. 
    \end{cases}
    \end{equation}
    For $|t|< C \log (h)$, it defines a unique Fourier integral operator (up to $h^{\infty} \, \Phi \mathrm{SO}_h^{-\infty} $) 
    \beqnx
    A_{{h}}(t_1,..., t_k) &=& e^{- \sum_{j=1}^k \frac{\mathrm{i} t_j}{{h}}   \mathrm{Op}_{f_j, {h}}  }  \, A_{{h}}(0) \, e^{ \sum_{j=1}^k \frac{\mathrm{i} t_j}{{h}}   \mathrm{Op}_{f_j, {h}}  }  + \mathcal{O}(h \, \Phi \mathrm{SO}_h^{\, { m}-1} )   \\
    &=&\mathrm{Op}_{a (t), {h}}  + \mathcal{O}(h \, \Phi \mathrm{SO}_h^{\,m-1} ). 
    \eqnx
\end{Proposition}

\begin{proof} 
First, by noting that $[ \mathrm{Op}_{a} , \mathrm{Op}_{b} ] = \mathrm{Op}_{ \{a,b\}} + \mathcal{O}(h  \Phi\mathrm{SO}_h^{m+n-2} )$ if $a \in S^{m} ( T^*(\partial D))$ and $b \in S^{n} ( T^*(\partial D))$ and using the given condition that $[\mathrm{Op}_{f_i, {h}},\mathrm{Op}_{f_j, {h}}] = 0$ whenever $i \neq j$, one can construct the symbol in the principal level. Then one can construct the full symbol in an inductive manner, and bounds the error operator via the Calder\'on-Vaillancourt theorem repeatedly. By Beals' theorem, the operator is guaranteed as a Fourier integral operator.  Explicit expression of $A_h(t)$ comes from checking the principal symbols, and bounding the error operator via the Zygmund trick.
\end{proof}

{\color{black} When $d \geq 3$,} we proceed to consider $ f_1 (x, \xi) = \tilde{H}(x,\xi) = \rho \left( [p_{\mathcal{K}^*_{\partial D}}(x,\xi)  ]^2 \right) $, and can immediately see that
\beqnx
\mathrm{Op}_{f_1,h} =\mathrm{Op}_{\tilde{H},h}  = \rho \left( [ \mathcal{K}^*_{h,\partial D} ]^2 \right) \text{ mod } (h \Phi\mathrm{SO}_h^{-3} ) \, . 
\eqnx
Let us also denote $\{L_{j,h}\}_{j=2}^k$ a family of pseudo-differential opererators such that 
\[
 \mathrm{Op}_{f_j, {h}}  = L_{j,h}  \, (h \Phi\mathrm{SO}_h^{ m} ) \,,
\]
then we immediately obtain the following corollary.
\begin{Corollary}
    \label{congugation} Under Assumption (A), {\color{black} when $d \geq 3$,} it holds that
    \beqnx 
    A_{{h}}(t) &=& e^{- \frac{\mathrm{i} t_1}{h}  \rho \left( [ \mathcal{K}^*_{h,\partial D} ]^2 \right)  - \sum_{j=2}^k \frac{\mathrm{i} t_j}{{h}}   L_{j,h}  }  \, A_{{h}}(0) \, e^{ \frac{\mathrm{i} t_1}{h}  \rho \left( [ \mathcal{K}^*_{h,\partial D} ]^2 \right)  - \sum_{j=2}^k \frac{\mathrm{i} t_j}{{h}} L_{j,h} }  + \mathcal{O}(h \, \Phi \mathrm{SO}_h^{\, m -1} ) \\
    &=&  \mathrm{Op}_{a (t), {h}}  , + \mathcal{O}(h \, \Phi \mathrm{SO}_h^{\, m-1} ). 
    \eqnx
\end{Corollary}

\subsection{Trace formula and generalized Weyl's law.}     We first state the Schwartz functional calculus as follows without proof. 
\begin{Lemma} \cite{Hor1,Hor2}
Recall that ${\mathcal{S}(\mathbb{R})}$ is the space of Schwartz functions on $\mathbb{R}$. Then for $f \in {{\mathcal{S}}}(\mathbb{R})$ and $a \in S^m(T^*(\partial D))$, we have $f ( \mathrm{Op}_{a,h} ) \in \Phi \mathrm{SO}_h^{-\infty}$ and 
\begin{equation}\label{eq:sfc1}
 f ( \mathrm{Op}_{a,h} ) = \mathrm{Op}_{f ( a )} + \mathcal{O}( h \Phi \mathrm{SO}_h^{-\infty}).
\end{equation}
\end{Lemma}

The above lemma leads us to the following trace theorem:
\begin{Proposition}\cite{Hor1,Hor2,trace,erg1,erg2}
Given $a \in S^m(T^*(\partial D))$, if $\mathrm{Op}_{a,h} $ is in the trace class and $f \in {\mathcal{S}(\mathbb{R})}$, then
\beqnx
(2 \pi h )^{ (d-1)} \mathrm{tr} ( f(  \mathrm{Op}_{a,h}  ) ) =   \int_{T^*(\partial D)} f( a ) \, d \sigma \otimes d \sigma^{-1} + \mathcal{O}(h) \, ,
\eqnx
where $d \sigma \otimes d \sigma^{-1}$ is the Liouville measure given by the top form $\omega^{d-1} / (d-1) !$.
\end{Proposition}

\noindent 
Let $(\lambda^i_1({h}), \cdots, \lambda^i_{k}({h}))$ be elements of the joint spectrum $\Sigma(\mathrm{Op}_{{h}}f_1, \cdots, \mathrm{Op}_{{h}}f_k)$ (for simplicity of notation, let us call it $\Sigma$), with the joint eigenstates $\phi^i({h})$; that is, $\mathrm{Op}_{{h}}f_j \phi^i({h}) = \lambda^i_j({h})\phi^i({h}), \ j \in \{1, \cdots, k\}.$  
We remark that, with this notation, it holds that  $( \tilde{\lambda}^i(h),  \phi^i(h) )$ is an eigenpair of $\mathcal{K}_{h,\partial D}$ if and only if  $( \lambda^i_{1}({h}) = \rho (\lambda^i(h) )^2 ,  \phi^i(h) )$ is an eigenpair of $\rho(  \mathcal{K}_{h,\partial D}^2  )$.
Then we can prove the following result. 

\begin{Proposition} 
Under assumption (A), {\color{black} when $d \geq 3$,} let $\mathcal{C} \subset \mathbb{R}^k$ be a compact convex polytope. Then  for any $a_j \in {\mathcal{S}}^m(T^{*}(\partial D)), j \in \{1, 2, \cdots, a_k\}$, we have as $h\rightarrow+0$,
\beqn
(2 \pi h )^{d-k} \sum_{(   \lambda_1^i(h)  , \cdots, \lambda_k^i(h)) \in \mathcal{C}} c_i  \, \langle  \mathrm{Op}_{a, {h}}  \, \phi^i(h) , \phi^i(h) \rangle_{L^2(\partial D, d \sigma)}  =  \int_{ \left\{F = (f_1, \cdots, f_k) \in \mathcal{C}\right\} } a  \, d \sigma \otimes d \sigma^{-1} + o_{\mathcal{C}}(1),
\label{weyl}
\eqn
where $c_i := | \phi^i |_{H^{-\frac{1}{2}}(\partial X, d \sigma)}^{-2} $ and the little-$o$ depends on $\mathcal{C}$.
\end{Proposition}
\begin{proof}
Let us first take $\mathcal{C} := \prod_{j =1}^k [r_j, s_j]$ a $k$ dimensional rectangle.
Take $$\chi_{1,\varepsilon} \left(\rho [ \mathcal{K}^*_{h,\partial D} ]^2 \right), \chi_{2,\varepsilon}\left( L_{2,h} \right), \cdots, \chi_{k,\varepsilon}\left( L_{k,h} \right) \,, $$ where $\chi_{\varepsilon}(x) := \prod_{j=1}^k  \chi_{j,\varepsilon} (x_j) \in {\tilde{\mathcal{S}}}(\mathbb{R}^{k})$ approximate $\chi_{ \prod_{j =1}^k [r_j, s_j]   }$. Then $\chi_{1,\varepsilon} \left(\rho [ \mathcal{K}^*_{h,\partial D} ]^2 \right) \in \Phi \text{SO}_h^{-\infty}$ and  $\chi_{j,\varepsilon}\left( L_{j,h} \right)\in \Phi \text{SO}_h^{-\infty}$, for each $j \in \{2, \cdots, k\}$ by the functional calculus with the trace formula 
\begin{equation}\label{eq:wyl1}
    \begin{split}
(2 \pi h )^{(d-k)}  \text{tr}  \left(  \chi_{1,\varepsilon} \left(\rho [ \mathcal{K}^*_{h,\partial D} ]^2 \right) \prod_{j=2}^k    \chi_{j,\varepsilon}\left(  L_{j,h} \right)  \, \text{Op}_{a,h}   \,  \chi_{1,\varepsilon} \left(\rho [ \mathcal{K}^*_{h,\partial D} ]^2 \right) \prod_{j=2}^k    \chi_{j,\varepsilon}\left( L_{j,h} \right)  \right) \\
=  \int_{T^*(\partial D)} a  \chi_{1,\varepsilon}^2(\rho(H) )\prod_{j=2}^k  \chi_{j,\varepsilon}^2 ( f_j ) \, d \sigma \otimes d \sigma^{-1} + \mathcal{O}_{\mathcal{C},\varepsilon}(h) \,,
    \end{split}
\end{equation}
where $\mathcal{O}$ depends on $\mathcal{C} ,\varepsilon$.
Passing $\varepsilon$ to $0$ in \eqref{eq:wyl1}, $  \chi_{1,\varepsilon} \left(\rho [ \mathcal{K}^*_{h,\partial D} ]^2 \right) \prod_{j=2}^k    \chi_{j,\varepsilon}\left(  L_{j,h} \right)  $ converges to the spectral projection operator when the joint spectrum $ (  \lambda^i_1({h}) , \cdots, \lambda^i_{k}({h}))  \in \mathcal{C} = \prod_{j =1}^k [r_j, s_j]$, which readily gives \eqref{weyl} when $\mathcal{C} = \prod_{j =1}^k [r_j, s_j]$.   Next we realize that a general compact convex polytope $\mathcal{C}$ is rectifyable, and can hence be approximated by artituary refinement of cover by a finite disjoint union of $k$ dimensional rectangles, and a standard approximation argument leads ust to \eqref{weyl} for a general $\mathcal{C}$. Finally, we notice that $
 \| \phi^i (h) \|_{L^{2}(X , d \sigma)} =  \| \phi^i \|_{H^{-\frac{1}{2}}(X, d \sigma)} $.
 
 The proof is complete. 
\end{proof}

If taking $a =1 $ in \eqref{weyl}, one readily has the classical Weyl's law \cite{Hor1,Hor2,trace,erg1,erg2} and \cite{ACL2}.
An algebraic proof of this result can also be found in \cite{Ngoc1, Ngoc2}.

\begin{Corollary}  Let $\mathcal{C} \subset \mathbb{R}^k$ be a compact convex polytope.
\beqn
 \sum_{( \lambda^i_1(h),  \cdots, \lambda^i_k(h)) \in \mathcal{C} } 1  = (2\pi{h})^{k-d} \int_{ \left\{ F = (f_1, \cdots, f_k) \in \mathcal{C} \right\} }  \, d \sigma \otimes d \sigma^{-1} + o_{\mathcal{C}}({h}^{-n}).
\label{weyl2}
\eqn
\end{Corollary}

\subsection{Ergodic decomposition theorem and quantum ergodicity on the leaves of the foliation by the integrable system.} 
From now on, denoting
\beqn
\mathcal{T} : T^*(\partial D) \times [0,\infty)^k & \rightarrow& T^*(\partial D) \notag \\
\mathcal{T} \left( (x,\xi), t \right) &=& (\gamma(t),p(t))
\label{Transformations}
\eqn
where $ (\gamma(\cdot),p(\cdot))$ is the \eqref{ODE_joint_flow},
we may adopt the notion in \cite{joint_ergodicity_1} in our case when $[0,\infty)^k$ forms a semi-group (we may in fact extend it to $\mathbb{R}^k$ by extending the Hamiltonian flow, but it is not necessary) and recall the definition of ergodicity for our purpose.
\begin{Definition}
Consider the family of maps $\{ \mathcal{T}_t (\cdot ) :=\mathcal{T} (\cdot,t)\}_{ t \in [0,\infty)^k}$.  Consider an invariant subspace $M \subset T^*(\partial D)$, i.e. $\mathcal{T}_t (M) = M$ for all $t \in [0,\infty)^k$.  We call a Radon measure over $M \subset T^*(\partial D)$ an invariant measure $\mu$ with respect to the joint Hamiltonian flow of the vector fields $X_{f_j}, j \in \{1, \cdots, k\}$ on $M$ if
\[
[ \mathcal{T}_t ]_{\#} \mu = \mu \text{ for all } t \in [0,\infty)^k,
\]
i.e. the push-foward measure of the measure coincide with itself.  We denote the set of such invariant measure as $M_{X_{F}}(M)$ (which forms a convex set.)
An invariant measure is ergodic with respect to the joint Hamiltonian flow generated by the vector fields $X_{f_j}, j \in \{1, \cdots, k\}$ on $M$ if for any measurable set $A \subset M$:
\[
\mu( A \Delta \mathcal{T}_t (A) ) = 0 \text{ for all } t \in [0,\infty)^k \ \quad \Longrightarrow \quad \mu(A) = 0 \text{ or }1\,.
\]
We denote the set of such ergodic measures as $M_{X_{F},\text{erg}}(M) $ (which can be directly checked to be the set of extremal points of $M_{X_{F}}(M)$.)
\end{Definition}
We remark that the standard equivalence of ergodicity, say e.g. for any measurable set $A \subset M$:
\[
\mathcal{T}_{t}^{-1} (A) \subset A \text{ for all } t \in [0,\infty)^k \ \quad \Longrightarrow \quad \mu(A) = 0 \text{ or }1\,,
\]
can be readily shown via standard and elementary arguments.
We would also like to remark that this definition of ergodicity is not that of the joint ergodicity of the family of commuting one parameter subgroups generated by each $X_{f_i}$ as introduced in e.g. \cite{joint_ergodicity_1,joint_ergodicity_2}, (which is instead a generalization of the mixing properties.)

Now, similar to our study in \cite{ACL2}, let us consider the set $\{F = (f_1, \cdots, f_k) = (\rho(1), e_2, \cdots, e_k)\}$ denoted by $F_{(1, e_2, \cdots, e_k)}$, for  the functions $f_j \in \mathcal{S}^m({T^*(\partial D)})$ defined as above, with $f_1 = \tilde{H}$. Let us denote $\sigma_{F_{(1, e_2, \cdots, e_k)}} $ as the Liouville measure on $F_{(1, e_2, \cdots, e_k)} \subset T^*(\partial D)$, and the same $\sigma_{F_{(E, e_2, \cdots, e_k)}} $ as that of $F_{(E, e_2, \cdots, e_k)} := \{F = (f_1, \cdots, f_k) = (\rho(E), e_2, \cdots, e_k)\} $ when $E\neq 1$.
For notational sake, from now on, we write 
\[
\mathcal{F} = \{ e := (e_2,...,e_k) \in \mathbb{R}^{k-1} \,:\, F_{(1,e_2,...,e_k)} \neq \emptyset \}\,.
\]
For all such $e \in \mathcal{F}$, since $X_{f_j} f_i = 0 $ and $\mathcal{L}_{X_{f_j}} \omega^{d-1} = 0 $ for all $i, j \in \{1, \cdots, k\}$,  {\color{black} we have that
$$ d \sigma_{F_{(1,e)}} := \lim_{\varepsilon \rightarrow 0} \, (2\varepsilon)^{ - k } \, \chi_{\bigcup_{v \in [-\varepsilon,\varepsilon]^k }F_{(1, e) + v }}   \, d \sigma \otimes d \sigma^{-1} $$ is an invariant measure of the (joint) flow on $ F_{(1, e)}  $.}

With the previous notion of ergodicity in hand, we next consider $M_{X_{F}}(F_{(1, e)})$, which is the set of invariant measures on  $F_{(1, e)}$, and $M_{X_{F},\text{erg}}(F_{(1, e)} ) $, which is the set of ergodic measures on $F_{(1, e)}$.
Since $F_{(1, e)}$ has a countable base, the weak-* topology of $M_{X_{F}}(F_{(1, e)})$ is metrizable, and hence Choquet's theorem can be applied to obtain the following  generalized version of the ergodic decomposition theorem in \cite{ergodic}.
\begin{Lemma} \label{full_decomposition}
Given a probability measure $\eta \in M_{X_{F}}(M)$, there exists a probability measure $\nu_e \in M ( M_{X_{F},\text{erg}}(F_{(1, e)})) $ such that
\[
\eta = \int_{M_{X_{F},\text{erg}}(F_{(1, e)})}  \mu_e \,  d \nu_e (\mu_e ) \,.
\]
\end{Lemma}
Applying Lemma~\ref{full_decomposition} to $\sigma_{F_{1,e}} / \sigma_{F_{(1,e)}} (F_{(1, e)})$, we have a probability measure $\nu_e \in M ( M_{X_{{F}},\text{erg}} (F_{(1, e)} )$ such that
\[
\sigma_{{F_{(1,e)}}} = \sigma_{{F_{(1,e)}}} (F_{(1, e)}) \int_{M_{X_{{F}},\text{erg}} (F_{(1, e)}) }  \mu_e \,  d \nu (\mu_e) \,.
\]
Note by rescaling $F_{(E, e)} = E^{-1/2} F_{(1, e)} $, we have $ \sigma_{{F_{(E,e)}}}  = E^{\frac{1+k-d}{2} }  \sigma_{{F_{(1,e)}}}  $.
Then, from the smoothness of $F$ and the non-degeneracy of $DF$ (up to a co-dimensional $1$ subset), we can see that the decomposition 
\[
d \sigma \otimes d \sigma^{-1} = \int_{\mathcal{F} } \left( E^{\frac{1+k-d}{2} }  d E \otimes d \sigma_{{F_{(1,e)}}}  \right) \varphi(e) \, de
\]
holds for some density $\varphi \in L^1(\mathcal{F} , de) $ (with $\{\varphi=0 \}$ of measure $0$ with respect to $de$) via a change of variable formula.

For any $\mu_e \in  M_{X_{{F}},\text{erg}}(F_{(1, e)})$, we let $\mu_{(E,e)} : = [ m_{ E^{-1/2} } ]_{\#} \mu_e  \in M_{X_{{f}},\text{erg}} (F_{(E, e)})$ be the push-forward measure given by
\[
m_{ E^{-1/2} } : T^*(\partial D) \rightarrow   T^*(\partial D)  \,,\, \quad  (x,\xi) \mapsto (x, E^{-1/2}  \xi)  \,. 
\] 
It then holds that
\[
 \sigma \otimes \sigma^{-1}   =
 \int_{\mathcal{F} } \int_{(0,\infty) \times M_{X_{{F}},\text{erg}} (F_{(1, e)}) }  \mu_{(E,e)}   \, \sigma_{{F_{(1,e)}}} (F_{(1, e)}) \, h(e)  E^{\frac{1+k-d}{2} }  \, ( d E \otimes
 \,  d \nu_e)  ( E, \mu_e )  
 \, de \,.
\]

Next, we shall derive a general version of the quantum ergodicity on the leaves of the foliation by a quantum integrable system.
Since $[0,\infty)^k$ is a countable amenable semi-group, ergodicity of a measure $\mu$ on $T^*(\partial D) $ is equivalent to that for all $f \in L^2 (T^*(\partial D), d \mu)$ (cf. \cite{joint_ergodicity_1}):
\[
\lim_{T \rightarrow 0} \frac{1}{\prod_{j=1}^k |T_j|}\int_{\prod_{j=1}^k [0,T_j]} f \circ \mathcal{T}_{t}  \, dt = \int_{T^*(\partial D)} f d \mu 
\]
in the $L^2 (T^*(\partial D), d \mu)$ metric,
which is in fact the Von-Neumann's ergodic theorem \cite{vonneumann} in this scenario.
Following \cite{ACL}, we have the following application from Birkhoff \cite{birkhoff} and Von-Neumann's ergodic theorems \cite{vonneumann}.

\begin{Lemma} 
 \label{birkroff} 
Under assumption (A), {\color{black} when $d \geq 3$,} for any $r_j\leq s_j, j \in \{1, \cdots, k\}$ and all $a_0 \in {\tilde{\mathcal{S}}}^m(T^*X)$,
we have as $T \rightarrow \infty$,
\beqnx
\displaystyle{\frac{1}{ \prod_{j=1}^k T_j} \int \limits_{\prod_{j=1}^k [0, T_j]} a_{(x,\xi)}(t) dt \rightarrow_{a.e. d \sigma \otimes d \sigma^{-1} \text{ and } L^2(\cap_{j =1}^k\{ r_j \leq f_j \leq s_j\}, d \sigma \otimes d \sigma^{-1} )}}\ \  \bar{a}(x,\xi)
\eqnx
for $\bar{a} \in L^2(\cap_{j =1}^k\{ r_j \leq f_j \leq s_j\}, d \sigma \otimes d \sigma^{-1} ) $ and a.e. $ ( d E \otimes
 \,  d \nu_e) ( E, \mu_e )  
 \, de $, with
\beqnx
\bar{a} (x,\xi ) = \int_{F_{(E, e)} } a_0   \, d \mu_{(E,e)}  \quad   \text{ a.e. } d \mu_{(E,e)} .
\eqnx
\end{Lemma}
\begin{proof}
By Birkhoff and Von-Neumann's ergodic theorems \cite{birkhoff,vonneumann} on $\chi_{\{\cap_{j=1}^k\{ r_j \leq f_j \leq s_j\}\} } \,  d \sigma \otimes d \sigma^{-1}  $, we have as $T \rightarrow \infty$:
\beqnx
\frac{1}{\prod_{j=1}^k |T_j|}\int_{\prod_{j=1}^k [0,T_j]} a_{(x,\xi)}(t) dt \rightarrow_{a.e. d \sigma \otimes d \sigma^{-1} \text{ and } L^2(\{\cap_{j=1}^k\{ r_j \leq f_j \leq s_j\}\}, d \sigma \otimes d \sigma^{-1} ) } \bar{a}(x,\xi),
\eqnx
for some  $\bar{a} \in L^2 ( \cap_{j=1}^k\{ r_j \leq f_j \leq s_j\} , d \sigma \otimes d \sigma^{-1} ) $ invariant under the joint Hamiltonian flow. 
Set 
\[ \mathcal{E} := \left \{ (x,\xi) \in \cap_{j=1}^k\{ r_j \leq f_j \leq s_j\}   : \limsup_T \left| \frac{1}{\prod_{j=1}^k |T_j|}\int_{\prod_{j=1}^k [0,T_j]} a_{(x,\xi)}(t) dt  - \bar{a}(x,\xi) \right | > 0   \right \} \, . \] 
It is clearly seen that $ \sigma \otimes \sigma^{-1} ( \mathcal{E} ) = 0$.  Next, we can show by Lemma \ref{full_decomposition} that
\beqnx
&&\hspace*{-3mm} \int \limits_{\mathcal{F}  \bigcap \prod_{j=2}^k [r_j, s_j] } \int \limits_{[r_1,s_1] \times M_{X_{{F}},\text{erg}} (F_{(1, e)}) }  \mu_{(E,e)} ( \mathcal{E} )     \, \sigma_{{F_{(1,e)}}} (F_{(1, e)}) \, \varphi(e)  E^{\frac{1+k-d}{2} }  \, ( d E \otimes
 \,  d \nu_e) ( E, \mu_e )  
 \, de \\
&&= \sigma \otimes \sigma^{-1} ( \mathcal{E} ) = 0 \,.
\eqnx
 Since $\{\varphi=0\}$ is of measure zero with respect to $de$, we have, for a.e. $ ( d E \otimes
 \,  d \nu_e) ( E, \mu_e )  
 \, de $, we have $\mu_{(E,e)} (\mathcal{E} ) = 0 $.  Meanwhile, by using the Birkhoff and Von-Neumann's ergodic theorems \cite{birkhoff,vonneumann} again, we have on each leaf that
\beqnx
\frac{1}{\prod_{j=1}^k |T_j|}\int_{\prod_{j=1}^k [0,T_j]} a_{(x,\xi)}(t) dt  \rightarrow_{a.e. \mu_{(E,e)} \text{ and } L^2(F_{(E,e)} , d \mu_{(E,e)} )} 
 \int_{F_{(E,e)} } a_0   \, d \mu_{(E,e)}  \quad  \text{ as $T \rightarrow \infty$.}
\eqnx
Finally, setting
\[
\begin{split} 
\mathcal{E}_{\mu_{(E,e)}} := &\Bigg \{ (x,\xi) \in \cap_{j=1}^k\{ r_j \leq f_j \leq s_j \}  :\\
 &\limsup_T \left |\frac{1}{\prod_{j=1}^k |T_j|}\int_{\prod_{j=1}^k [0,T_j]} a_{(x,\xi)}(t) dt  -  \int_{F_{(E,e)} } a_0   \, d \mu_{(E,e)} \right | > 0   \Bigg \} \,, 
\end{split}
\]
we can show that $\mu_{(E,e)} ( \mathcal{E}_{\mu_{(E,e)}}  ) = 0 $.  Therefore, a.e. $  ( d E \otimes
 \,  d \nu_e) ( E, \mu_e )  
 \, de  $, $\mu_{(E,e)} \left(\mathcal{E} \bigcup \mathcal{E}_{\mu_{(E,e)} } \right) = 0 $.  By the uniqueness of the limit, the proof can be readily concluded.
\end{proof}


With the above preparations, we can establish the following theorem that shall play an important role in our subsequent analysis. 

\begin{Theorem}\label{thm:ergo1}
  Let $\mathcal{C} \subset \mathbb{R}^k$ be a compact convex polytope and $c_i := | \phi^i |_{H^{-\frac{1}{2}}(X, d \sigma)}^{-2} $. Then under assumption (A), {\color{black} when $d \geq 3$,} the following (variance-like) estimate holds as $h\rightarrow+0$:
\begin{equation}\label{eq:est1}
\begin{split}
& \frac{1}{ \sum_{ (  \lambda^i_1(h) , \cdots, \lambda^i_k(h) )\in \mathcal{C}} 1 } \sum_{  ( \lambda^i_1(h)  , \cdots, \lambda^i_k(h) )\in \mathcal{C}} 
c_i^2 \bigg|   \langle  A_h \,  \phi^i(h) ,  \phi^i(h) \rangle_{L^2(\partial D , d \sigma)} \\
&\hspace*{4cm} -  \langle \mathrm{Op}_{\bar{a}, h}   \,  \phi^i(h) ,  \phi^i(h) \rangle_{L^2(\partial D , d \sigma)}   \bigg|^2 \rightarrow 0.
\end{split}
 \end{equation}
\end{Theorem}

\begin{proof}
Via considering the Hamiltonian flow of the principle symbol, we can lift the Birkhoff and Von-Neumann to the operator level.    Set $A_{h}(0) = A_{h}$. From the definition of $\phi^i(h)$, we have, for each $i$,
\begin{equation}\label{eq:est2}
\begin{split}
 & \langle A_{h}(t) \phi^i(h) , \phi^i(h) \rangle_{L^2(\partial D, d \sigma)}  \\
=& \langle  A_{{h}}(0) \, e^{ \frac{\mathrm{i} t_1}{h}  \rho \left( [ \mathcal{K}^*_{h,\partial D} ]^2 \right)  - \sum_{j=2}^k \frac{\mathrm{i} t_j}{{h}} L_{j,h} }   \phi^i(h) , \, e^{ \frac{\mathrm{i} t_1}{h}  \rho \left( [ \mathcal{K}^*_{h,\partial D} ]^2 \right)  - \sum_{j=2}^k \frac{\mathrm{i} t_j}{{h}} L_{j,h} }  \phi^i(h) \rangle_{L^2(\partial D, d \sigma)}   + \mathcal{O}_t (h) \\
=&  \langle A_{h}\,  \phi^i(h) ,  \phi^i(h) \rangle_{L^2(X, d \sigma)}+ \mathcal{O}_t (h) ,
\end{split}
\end{equation}
where we make use of Proposition \ref{prop:4.1} as well as the definition of the NP eigenfunctions in \eqref{eq:eg1}.  Averaging both sides of \eqref{eq:est2} with respect to $T$, we can arrive at
\beqnx
 \left \langle  \Upsilon \phi^i(h) , \phi^i(h) \right \rangle_{L^2(\partial D, d \sigma)}  =  \langle A_{h}   \phi^i(h) ,  \phi^i(h) \rangle_{L^2(X , d \sigma)} + \mathcal{O}_T (h), 
\eqnx
where
\[
\Upsilon:=\frac{1}{ \prod_{j=1}^k |T_j|}\int \limits_{\prod_{j=1}^k [0,T_j] }  A_{h}(t) dt.
\]
By using Proposition \ref{prop:4.1} again, one can directly verify that
\beqnx
\frac{1}{ \prod_{j=1}^k |T_j|}\int \limits_{\prod_{j=1}^k [0,T_j] }  A_{h}(t) dt  -   \text{Op}_{\bar{a}, h} = \mathrm{Op}_{ \frac{1}{ \prod_{j=1}^k |T_j|}\int \limits_{\prod_{j=1}^k [0,T_j] }  a_{(x,\xi)}(t) dt  - \bar{a} } + \mathcal{O}_T (h ). 
\eqnx
Next by using the Cauchy-Schwarz inequality, we have
\begin{equation}\label{eq:sum1}
\begin{split}
 &\left|  \frac{ \left \langle   \text{Op}_{\bar{a}, h}  \phi^i(h) , \phi^i(h) \right \rangle_{L^2(\partial D, d \sigma)} }{{\langle \phi^i(h) , \phi^i(h) \rangle_{L^2(\partial D, d \sigma)}} } - \frac{ \langle  A_{{h}}     \phi^i(h) ,  \phi^i(h) \rangle_{L^2(\partial D, d \sigma)} }{{\langle \phi^i(h) , \phi^i(h) \rangle_{L^2(\partial D, d \sigma)}} } \right|^2 \\
\leq & \frac{  \left \langle  \Xi^* \Xi  \phi^i(h) ,  \phi^i(h) \right \rangle_{L^2(\partial D, d \sigma)} }{{\langle \phi^i(h) , \phi^i(h) \rangle_{L^2(\partial D, d \sigma)}} } + \mathcal{O}_T (h^2),  
\end{split}
\end{equation}
with
\[
\Xi:=\Upsilon  -   \text{Op}_{\bar{a}, h}.
\]
Therefore, summing up over the joint spectrum of $\phi^i$ of \eqref{eq:sum1} and applying \eqref{weyl} and \eqref{weyl2}, we have
\begin{equation}\label{eq:ee4}
\begin{split}
& \frac{1}{ \sum_{ ( | \lambda^i_1(h) | , \cdots, \lambda^i_k(h) )\in \mathcal{C}} 1 }
  \sum_{ ( | \lambda^i_1(h) | , \cdots, \lambda^i_k(h) )\in \mathcal{C}}   c_i^2 \bigg|   \langle  A_h \,  \phi^i(h) ,  \phi^i(h) \rangle_{L^2(\partial D , d \sigma)}\\
  &\hspace*{4cm}  -  \langle \text{Op}_{\bar{a}, h}   \,  \phi^i(h) ,  \phi^i(h) \rangle_{L^2(\partial D , d \sigma)}   \bigg|^2 \\
\leq&
\frac{\int_{ \{F = (f_1, \cdots, f_k )\in \mathcal{C}\} } \left| \frac{1}{ \prod_{j=1}^k |T_j|}\int \limits_{\prod_{j=1}^k [0,T_j] }  a_{(x,\xi)}(t) dt  - \bar{a} \right|^2 \, d \sigma \otimes d \sigma^{-1}}{\int_{ \{F = (f_1, \cdots, f_k )\in \mathcal{C}\} } \, d \sigma \otimes d \sigma^{-1}} + o_{\mathcal{C},T}(1).
\end{split}
\end{equation}
Finally, by noting that the first term at the right-hand side of \eqref{eq:ee4} goes to zero as $T = (T_1,...,T_k)$ goes to infinity, one can readily have \eqref{eq:est1}, which completes the proof.  
\end{proof}

With Theorem~\ref{thm:ergo1}, together with Chebeychev's trick and a diagonal argument, we can have the following quantum ergodicity result, {\color{black} which generalizes the relevant results in \cite{erg1,erg11,Toth02,Toth13,erg32,erg33,erg34,erg35,erg2,erg_a,erg_b,trace}.}
\begin{Corollary}
  Let $\mathcal{C} \subset \mathbb{R}^k$ be a compact convex polytope. Under assumption (A), {\color{black} when $d \geq 3$,} there exists $S(h) \subset J(h) := \{i \in \mathbb{N} : (\lambda^i_1(h), \cdots, \lambda^i_k(h)) \in \mathcal{C}   \}$ such that 
for all $a_0 \in S^{m}(T^*X)$, we have as $h\rightarrow +0$,
\beqn
\max_{i \in S(h)} c_i \left|   \left \langle ( A_h - \mathrm{Op}_{\bar{a}, h} )  \,  \phi^i(h) ,  \phi^i(h) \right \rangle_{L^2(X , d \sigma)} \right| = o_{\mathcal{C}}(1) \, \text{ and } \, \frac{\sum_{i \in S(h) } 1 }{ \sum_{i \in J(h) } 1 } = 1 + o_{\mathcal{C}}(1) \,.
\label{ergodicity}
\eqn
It is noted that the choice of $S(h) $ is independent of $a_0$. 
\end{Corollary}

{\color{black} To our best knowledge, the first time when a quantum integrable system is considered to show eigenfunction concentration on Lagrangian submanifolds is in \cite{Toth02}, where the Laplacian eigenfunctions are discussed instead.}

\begin{Remark}
It is indeed a bit paradoxical to still call the above theorem as that of quantum ergodicity when we now have a quantum integrable system. Since as is customarily understood, the case of (complete) integrable system and that of ergodicity is almost on the opposite side of the spectrum in the description of a dynamical system.
However, our discussion is on the ergodicity on the leaves of the foliation given by the integrable system (e.g. the Lagrangian tori if we have a complete integrable system), and therefore no paradox emerges.
\end{Remark}

\section{Locolization/concentration of plasmon resonances in electrostatics}

In this section, we are in a position to present one of our main results on the localization/concentration of plasmon resonances in electrostatics. 

\subsection{Consequences of generalized Weyl's law and quantum ergodicity}  
In the following, we let $\sigma_{x,F_{(1,e)}}$ signify the Liouville measure on $ F_{(1,e)}(x) := \{F(x, \cdot) = (\rho(1), e ) \} \subset T^*_x (\partial D)$. By the generalized Weyl's law in Section~\ref{sect:3}, we can obtain the following result, which characterizes the local behaviour of the NP eigenfunctions and their relative magnitude.

\begin{Theorem}
\label{theorem1}
Given any $x \in \partial D$, we consider $ \{ \chi_{x,\delta} \}_{\delta > 0}$ being a family of smooth nonnegative bump functions compactly supported in $B_{\delta} (x)$ with $ \int_{\partial D} \chi_{p,\delta} \, d \sigma = 1$.
Under Assumption (A), {\color{black} when $d \geq 3$,} fixing a compact convex polytope $ \mathcal{C}  \subset \mathcal{F} \subset \mathbb{R}^{k-1}$, $[r,s] \subset{\mathbb{R}}$, $\alpha\in\mathbb{R}$ and $p, q \in \partial D$, there exists a choice of $\delta(h)$ depending on $\mathcal{C},p,q, r,s$ and $\alpha$ such that, as $h\rightarrow +0$, we have $\delta(h) \rightarrow 0$ and
\beqn\label{concentration1}
\begin{split}
&\frac{ \sum_{(\lambda^i_1 (h), \cdots, \lambda^i_k (h)) \in [r,s] \times \mathcal{C}} c_i  \int_{\partial D} \chi_{p,\delta (h) }(x) | | D |^{\alpha}   \, \phi^i (x) |^2 d \sigma(x) }{ \sum_{(\lambda^i_1 (h), \cdots, \lambda^i_k (h)) \in [r,s] \times \mathcal{C}} c_i  \int_{\partial D} \chi_{q,\delta (h) }(x) | | D |^{\alpha}   \, \phi^i (x) |^2 d \sigma(x) }\medskip\\
 =&  \frac{     \int_{\mathcal{F} } \int_{ F_{(1,e)}(p)  }  | \xi |_{g(y)}^{1+ 2 \alpha } d \sigma_{p,F_{(1,e)}} \, h(e) d e  }{   \int_{\mathcal{F} } \int_{F_{(1,e)}(q)  }  | \xi |_{g(y)}^{1+ 2 \alpha } d \sigma_{q,F_{(1,e)}} \, h(e) d e   } + o_{\mathcal{C},r,s,p,q,\alpha}(1),
\end{split}
\eqn
where $c_i := | \phi^i |_{H^{-\frac{1}{2}}(\partial D , d \sigma)}^{-2} $.  In particular, if $\alpha = - \frac{1}{2}$, the RHS term of \eqref{concentration1} is the ratio between the volumes of $\bigcup_{e \in \mathcal{F}} F_{(1,e)}(\cdot) $ at the respective points.
\end{Theorem}
\begin{proof}
Taking $p\in \partial D$, we consider $a(x,\xi) := \chi_{p,\delta}(x)  | \xi |^{1+ 2 \alpha}_{g(x)}  $ in \eqref{weyl}.
With the fact that $\mathrm{Op}_{a,h} = h^{1+ 2 \alpha}  | D |^{1/2+ \alpha} \mathrm{Op}_{\chi_{p,\delta}(x),h} | D |^{1/2+ \alpha} - h \mathrm{Op}_{\tilde{a}_{p,\delta},h} $ for some $\tilde{a}_{p,\delta} \in S^{2 \alpha}(T^*(\partial D)) $, we have after applying \eqref{weyl} once more upon $\tilde{a}_{p,\delta}$:
\begin{equation}\label{eq:nnn1}
\begin{split}
&  (2 \pi h )^{ (d  + 2 \alpha ) }  \sum_{(\lambda^i_1 (h), \cdots, \lambda^i_k (h)) \in [r,s] \times \mathcal{C}} c_i  \int_{\partial D} \chi_{p,\delta}(x) | | D |^{\alpha}   \, \phi^i (x) |^2 d \sigma(x)  \\
= & \int_{ \{ (f_1, \cdots, f_k) \in [r,s] \times \mathcal{C} \} } \chi_{p,\delta}(x)  | \xi |^{1+ 2 \alpha}_{g(x)}  \, d \sigma \otimes d \sigma^{-1} \\
&\qquad\qquad+  h \int_{ \{(f_1, \cdots, f_k) \in [r,s] \times \mathcal{C} \} } \tilde{a}_{p,\delta} \, d \sigma \otimes d \sigma^{-1} + o_{\mathcal{C}, r,s,\alpha}(1). 
\end{split}
\end{equation}
With \eqref{eq:nnn1}, we have, after choosing another point $q \in \partial D$ and taking a quotient between the two, that
\beqnx
& & \frac{ \sum_{(\lambda^i_1 (h), \cdots, \lambda^i_k (h)) \in [r,s] \times \mathcal{C}} c_i  \int_{\partial D} \chi_{p,\delta}(x) | | D |^{\alpha}   \, \phi^i (x) |^2 d \sigma (x) }{ \sum_{(\lambda^i_1 (h), \cdots, \lambda^i_k (h)) \in [r,s] \times \mathcal{C}} c_i  \int_{\partial D} \chi_{q,\delta}(x) | | D |^{\alpha}   \, \phi^i (x) |^2 d \sigma (x) } \\
&=&  \frac{ \int_{ \{ (f_1, \cdots, f_k) \in [r,s] \times  \mathcal{C} \} } \chi_{p,\delta}(x)  | \xi |^{1+ 2 \alpha}_{g(x)}  \, d \sigma \otimes d \sigma^{-1} +  h \int_{ \{ (f_1, \cdots, f_k) \in [r,s] \times \mathcal{C} \} } \tilde{a}_{p,\delta} \, d \sigma \otimes d \sigma^{-1} }{ \int_{ \{ (f_1, \cdots, f_k) \in [r,s] \times \mathcal{C} \} } \chi_{q,\delta}(x)  | \xi |^{1+ 2 \alpha}_{g(x)}  \, d \sigma \otimes d \sigma^{-1} +  h \int_{ \{ (f_1, \cdots, f_k) \in [r,s] \times \mathcal{C} \} } \tilde{a}_{q,\delta}\, d \sigma \otimes d \sigma^{-1}} + o_{\mathcal{C},r,s,\alpha}(1).  
\eqnx
Now, for any given $h$, we can make a choice of $\delta (h)$ depending on $\mathcal{C},r,s,p,q,\alpha$ such that as $h \rightarrow +0$, we have $\delta (h) \rightarrow 0$ (much slower than $h$) and 
\[
\left| h \int_{ \{ (f_1, \cdots, f_k) \in [r,s] \times \mathcal{C} \} } \tilde{a}_{p,\delta (h)} \, d \sigma \otimes d \sigma^{-1}  \right| 
+  \left| h \int_{ \{(f_1, \cdots, f_k) \in [r,s] \times \mathcal{C} \} } \tilde{a}_{q,\delta (h)} \, d \sigma \otimes d \sigma^{-1} \right|  \rightarrow 0 \,.
\]
We also realize as $h \rightarrow +0$, with this choice of $\delta (h)$ that $\delta (h) \rightarrow 0$, and one has for $y = p,q$ that
\[
 \int_{ \{(f_1, \cdots, f_k) \in [r,s] \times \mathcal{C} \} } \chi_{y,\delta(h) }(x)  | \xi |^{1+ 2 \alpha}_{g(x)}  \, d \sigma \otimes d \sigma^{-1}  \rightarrow  \int_{ \{(f_1(y, \cdot), \cdots, f_k(y, \cdot))\in [r,s] \times \mathcal{C}\} } | \xi |^{1+ 2 \alpha}_{g(y)}  d \sigma^{-1} \,.
\]
Therefore, we have
\beqnx
&&\frac{ \sum_{(\lambda^i_1 (h), \cdots, \lambda^i_k (h)) \in [r,s] \times \mathcal{C}} c_i  \int_{\partial D} \chi_{p,\delta(h) }(x) | | D |^{\alpha}   \, \phi^i (x) |^2 d \sigma(x) }{ \sum_{(\lambda^i_1 (h), \cdots, \lambda^i_k (h)) \in [r,s] \times \mathcal{C}} c_i  \int_{\partial D} \chi_{q,\delta (h) }(x) | | D |^{\alpha}   \, \phi^i (x) |^2 d \sigma(x) }\\
& = & \frac{   \int_{ \{(f_1(p, \cdot), \cdots, f_k(p, \cdot))\in [r,s] \times \mathcal{C}\} } | \xi |^{1+ 2 \alpha}_{g(p)}  d \sigma^{-1}  }{   \int_{ \{ (f_1(q, \cdot), \cdots, f_k(q, \cdot))\in [r,s] \times \mathcal{C}\} } | \xi |^{1+ 2 \alpha}_{g(q)}  d \sigma^{-1}  } + o_{\mathcal{C}, r,s,p,q,\alpha}(1)  \,.
\eqnx
To conclude our proof, we realize that for all $y = p,q$,
\[
\begin{split}
 &\int_{ \{ (f_1(y, \cdot), \cdots, f_k(y, \cdot))\in [r,s] \times \mathcal{C}\} } | \xi |^{1+ 2 \alpha}_{g(y)}  d \sigma^{-1}\\
  =& \left(  \int_{r}^s E^{-\frac{k}{2} - \alpha - \frac{d}{2} } d E \right) \left( \int_{\mathcal{F} } \int_{  F_{(1,e)}(y)  }  | \xi |_{g(y)}^{1+ 2 \alpha } d \sigma_{y,F_{(1,e)}} \, h(e) d e \right),
\end{split}
\]
where we recall $h \in L^1(\mathcal{F} , de) $ and $d-1-k$ is the generic dimension of $F_{(1,e)}(y)$.

The proof is complete. 
\end{proof}

Theorem~\ref{theorem1} states that, given $p,q \in \partial D$, the relative magnitude between a $c_i$-weighted sum of a weighed average of $| |D|^{\alpha} \phi^i |^2$ over a small neighborhood of $p$ to that of $q$ asymptotically depends on the ratio between the weighted volume of $ \{ (f_1(p, \cdot), \cdots, f_k(p, \cdot))= (\rho(1), e) \,, e \in \mathcal{C} \} $ and that of $ \{ (f_1(q, \cdot), \cdots, f_k(q, \cdot))= (\rho(1), e)  \,, e \in \mathcal{C}  \}  $. This is critical for our subsequent analysis since it reduces our study to analyzing the aforementioned weighted volumes. 


\begin{Theorem}
\label{theorem2}
Under Assumption (A), {\color{black} when $d \geq 3$,} there is a family of distributions:
 \[
 \{ \Phi_{\mu, e}  \}_{\mu \in M_{X_{F},\text{erg}}(F_{(1,e)} ) , e \in   \mathcal{F}  } \in \mathcal{D}' ( \partial D \times \partial D),
 \] 
which are the Schwartz kernels of $\mathcal{K} _{\mu_e}$ such that they form a partition of the identity operator $Id$ as follows:
\begin{equation}\label{eq:pu1}
{Id =    \int_{\mathcal{F} } \int_{M_{X_{F},\text{erg}}( F_{(1,e)} ) }  \mathcal{K}_{\mu_e} \,  d \nu_e \, (\mu_e ) \,  de (e) \, .}
\end{equation}
It holds in the weak operator topology satisfying that for any given compact convex polytope $\mathcal{C} \subset (0,\infty) \times \mathcal{F} \subset \mathbb{R}^{k}$, there exists
$S(h) \subset J(h) := \{i \in \mathbb{N} : (\lambda^i_1(h), \cdots, \lambda^i_k(h) )\in \mathcal{C}  \}$ such that for all $\varphi \in C^\infty ( \partial D)$ and as $h\rightarrow +0$,
\begin{equation}\label{eq:pu2}
\begin{split}
&\max_{i \in S(h)}  \Bigg|   \int_{ \partial D}  \varphi(x) \bigg( c_i \, | |D|^{-\frac{1}{2}} \phi^i (x) |^2\\
 &-\int_{\mathcal{F} } \int_{M_{X_{F},\text{erg}}( F_{(1,e)} )}   \mu(x,e)     \, g_{i} (\mu_e)  \,  d\nu_e \, (\mu_e  )  \, de (e)  \bigg) d \sigma(x)
 \Bigg| = o_{\mathcal{C}}(1)  \,.
 \end{split}
\end{equation}
In \eqref{eq:pu1},
\begin{equation}\label{eq:pu3}
\begin{split}
&\qquad g_{i} (\mu_e)  :=  c_i 
 \langle  \mathcal{K}_{\mu_e}   |D|^{-\frac{1}{2}} \phi^i   , |D|^{-\frac{1}{2}} \phi^i  \rangle_{L^2(\partial D , d \sigma)}
\,, \\
&   \int_{\mathcal{F} } \int_{M_{X_{F},\text{erg}}( F_{(1,e)} ) }   g_{i} (\mu_e) \,  d \nu_e \, (\mu_e ) \,  de (e) = 1 \, , \, \
\frac{\sum_{i \in S(h) } 1 }{ \sum_{i \in J(h) } 1 } = 1 + o_{\mathcal{C}}(1) \,,
\end{split}
\end{equation}
and moreover,
\begin{equation}\label{eq:pu4}
\begin{split}
 &\qquad\quad \mu (p,e)  \geq 0 \,, \, \int_{\partial \Omega} \mu(p,e) d \sigma(p)  = 1\,, \, \\
  & \frac{\int_{M_{X_{F},\text{erg}}(F_{(1,e)})}  \mu(p,e)  \, d \nu_e(\mu_e)  }{  \int_{M_{X_{F},\text{erg}}(F_{(1,e)})}  \mu(q,e)  \, d \nu_e(\mu_e) }  =  \frac{ \int_{F_{(1,e)}(p) } d \sigma_{p,F_{(1,e)}}  }{  \int_{F_{(1,e)}(q) } d \sigma_{q,F_{(1,e)}}  } \text{ a.e. } (d \sigma \otimes d \sigma) (p,q) \,.
\end{split}
\end{equation}
\end{Theorem}

\begin{proof}
Let $f, \varphi \in \mathcal{C}^{\infty}(\partial D)$ be given.  Let us consider $a(x,\xi) := \varphi(x) $. Then we have 
\[
 \int_{ F_{(E, e)} } \varphi \, d \mu_{E,e}  =  \int_{F_{(1, e)}  } \varphi d \mu_{e} .
 \]  
Take a partition of unity $\{\chi_i\}$ on $\{U_i\}$. With an abuse of notation via identification of points with the local trivialisation $\{F_i\}$, by Lemmas \ref{full_decomposition} and \ref{birkroff}, we have
\begin{equation}\label{eq:pp1}
\begin{split}
& [ \text{Op}_{\bar{\varphi} , h} f ](y) \\
= &  \int_{\mathcal{F} } \int_{(0,\infty) \times M_{X_{{F}},\text{erg}} (F_{(1, e)}) } \sum_l   \left(  \int_{F_{(E,e)}}  \exp( \langle x-y, \xi \rangle / h )  \bar{a}(x,\xi)   \chi_l (x) f (x)  d \mu_{(E,e)}(x)  \right)   \\
&\hspace*{4cm}\times  \sigma_{{F_{(1,e)}}} (F_{(1, e)}) \, h(e)  E^{\frac{1 + k -  d}{2} } \, ( d E \otimes
 \,  d \nu_e)  ( E, \mu_e )  
 \, de (e) \,. \\
= &  \int_{\mathcal{F} } \int_{(0,\infty) \times M_{X_{{F}},\text{erg}} (F_{(1, e)}) } \sum_l   \left(  \int_{F_{(1, e)}  } \varphi d \mu_{e}  \right)    \left(  \int_{F_{(1,e)}}  \exp( \langle  x-y, E^{ - \frac{1}{2}} \xi \rangle / h )  \chi_l (x) f (x)  d \mu_{e}(x)  \right)   \\
&\hspace*{4cm}\times  \sigma_{{F_{(1,e)}}} (F_{(1, e)}) \, h(e)  E^{\frac{1 + k -  d}{2} } \, ( d E \otimes
 \,  d \nu_e)  ( E, \mu_e )  
 \, de (e) \,. 
\end{split}
\end{equation}
On the other hand, considering $Id = \mathrm{Op}_{1,h} = \text{Op}_{ \bar{1}, h}$ (which is independent of $h$), one can show that
\begin{equation}
    \begin{split}
&  [ \text{Op}_{ 1, h} f ](y)  \\
= &  \int_{\mathcal{F} } \int_{(0,\infty) \times M_{X_{{F}},\text{erg}} (F_{(1, e)}) } \sum_l   \left(  \int_{F_{(E,e)}}  \exp( \langle x-y, \xi \rangle / h )   \chi_l (x) f (x)  d \mu_{(E,e)}(x)  \right)   \\
&\hspace*{4cm}\times  \sigma_{{F_{(1,e)}}} (F_{(1, e)}) \, h(e)  E^{\frac{1 + k -  d}{2} } \, ( d E \otimes
 \,  d \nu_e)  ( E, \mu_e )  
 \, de (e) \,. \\
= &  \int_{\mathcal{F} } \int_{(0,\infty) \times M_{X_{{F}},\text{erg}} (F_{(1, e)}) } \sum_l   \left(  \int_{F_{(1,e)}}   \exp( \langle  x-y, E^{- \frac{1}{2} } \xi \rangle / h )  \chi_l (x) f (x)  d \mu_{e}(x)  \right)  \, \\
&\hspace*{4cm}\times   \sigma_{{F_{(1,e)}}} (F_{(1, e)}) \, h(e)  E^{\frac{ 1 + k- d}{2} } \, ( d E \otimes
 \,  d \nu_e)  ( E, \mu_e )  
 \, de (e) \,.
\end{split}
\end{equation}
If defining $ \mathcal{K}_{\mu}$ (which is again independent of $h$) to be such that
\beqnx
&&  [ \mathcal{K}_{\mu_{e}} f ](y)  := \int_{(0,\infty) } \sum_l   \left(  \int_{F_{(1,e)}}   \exp( \langle x-y,  E^{- \frac{1}{2} }  \xi \rangle / h )  \chi_l (  x) f (x)  d \mu_{e}(x)  \right)\\
  &&\hspace*{2cm}\times \sigma_{{F_{(1,e)}}} (F_{(1, e)}) \, h(e)  E^{\frac{ 1 + k- d}{2} } \,  d E  ( E  )  \,,
\eqnx
we have by definition in the weak operator topology that
\[
Id =  \int_{\mathcal{F} } \int_{M_{X_{F},\text{erg}}( F_{(1,e)} ) }  \mathcal{K}_{\mu_e} \,  d \nu_e \, (\mu_e )  \, de (e) \, 
\]
That is, 
\[
 \langle f , f  \rangle_{L^2(\partial D , d \sigma)} = 
\int_{\mathcal{F} } \int_{M_{X_{F},\text{erg}}( F_{(1,e)} ) }  \langle  \mathcal{K}_{\mu_e}   f , f  \rangle_{L^2(\partial D , d \sigma)} \,  d \nu_e \, (\mu_e )  \, de (e) \,,
\]
whereas
\beqnx
 \langle \text{Op}_{  \overline{ \varphi }  , h} f , f  \rangle_{L^2(\partial D , d \sigma)}  =  
\int_{\mathcal{F} } \int_{M_{X_{F},\text{erg}}( F_{(1,e)} ) }  \left(  \int_{F_{(1, e)}  } \varphi d \mu_{e}  \right)   \, \langle  \mathcal{K}_{\mu_e}   f , f  \rangle_{L^2(\partial D , d \sigma)} \,  d \nu_e \, (\mu_e )  \, de \,.
\eqnx
Recall that $d \sigma_{F_{(1,e)}}(x,\xi) / \sigma_F ( F_{(1, e)} ) =  d \mu_e (x,\xi) \,  d \nu_e (\mu_e) $ is a probability measure. We now apply the disintegration theorem to the measure $d \mu_e (x,\xi) \,  d \nu_e (\mu_e) $ and obtain a disintegration $d \mu_{p,e} (x,\xi) \, d\nu_e(\mu_e) \otimes d \sigma(p) $, where the measure-valued map $ ( \mu_e , p ) \mapsto \mu_{p,e}$ is a $d\nu_e \otimes d \sigma $ measurable function together with $\mu_{p,e} \left( F_{(1, e) }  \backslash ( F_{(1, e) }(p) \bigcap \text{spt}(\mu_e) ) \right) = 0 $ a.e. $d\nu_e \otimes d \sigma$.  Therefore, we obtain
\beqnx
&& \langle \text{Op}_{  \overline{ \varphi }  , h} f , f  \rangle_{L^2(\partial D , d \sigma)}\\
  &= & 
\int_{\mathcal{F} } \int_{M_{X_{F},\text{erg}}( F_{(1,e)} )  \times \partial D}  \int_{F_{(1, e)}  } \varphi   \, \langle  \mathcal{K}_{\mu_e}   f , f  \rangle_{L^2(\partial D , d \sigma)} \,  
d \mu_{p,e} \,  (d\nu_e \otimes d \sigma) \, (\mu_e ,p )  \, de (e) \,.
\eqnx
It is also observed that
\[
\int_{F_{(1,e)} }  \varphi \, d \mu_{p,e}  =  \int_{ F_{(1,e)}(p) }  \varphi \, d \mu_{p,e} = \varphi(p) \, \mu_{p,e}(  F_{(1,e)}  ) .
\]
If we denote
\[
\mu(p,e) := \mu_{p,e}(  F_{(1,e)}  )  \geq 0 \,,
\]
then a.e. $d \nu_e (\mu_e)$, the function $\mu( \cdot , e) \in L^1(\partial \Omega, d \sigma)$. As a result of the disintegration, we have a.e. $d \nu_e (\mu_e)$,
\[
\int_{\partial \Omega} \mu(p,e) d \sigma(p) = \mu_e ( F_{(1,e)}) = 1 \,.
\]
Furthermore, we have
\beqnx
&& \langle \text{Op}_{  \overline{ \varphi }  , h} f , f  \rangle_{L^2(\partial D , d \sigma)}\\
 & =&  
\int_{\mathcal{F} } \int_{M_{X_{F},\text{erg}}( F_{(1,e)} )  \times \partial D}   \varphi(x) \, \mu(x,e)     \, \langle  \mathcal{K}_{\mu_e}   f , f  \rangle_{L^2(\partial D , d \sigma)} \,  (d\nu_e \otimes d \sigma) \, (\mu_e ,x )  \, de (e)  \,.
\eqnx

Finally, we choose $f = \phi^i (h) = |D|^{-\frac{1}{2}} \phi^i$ and apply \eqref{ergodicity} to obtain the conclusion of our theorem. It is noted that the choice of $S(h)$ is independent of $\varphi \in \mathcal{C}^{\infty} (\partial D)$.
The ratio in the last line of the theorem comes from the fact that a.e. $ d \sigma(p)$ we have by definition
\begin{equation}
\begin{split}
    \int_{M_{X_{F},\text{erg}}(F_{(1,e)})}  \mu(p,e)  \, d \nu_e(\mu_e) :=&    \int_{M_{X_{F},\text{erg}}(F_{(1,e)})}  \mu_{p,e} ( F_{(1,e)}(p))  \, d \nu_e(\mu_e)  \\ =&  \frac{ \int_{F_{(1,e)}(p) } d \sigma_{p,F_{(1,e)}}  }{  \int_{F_{(1,e)} } d \sigma_{ F_{(1,e)}}   } \,.
\end{split}
\end{equation}
 
The proof is complete. 
\end{proof}

Theorem~\ref{theorem2} indicates that most of the function $ c_i ||D|^{-\frac{1}{2}} \phi^i (x) |^2$ weakly converges to a $g_{i} (\mu_e) \, d \nu_e(\mu_e)$-weighted average of  $ \mu(x,e) $ on each leaf $F_{(1,e)}$,
where the ratio between a $d \nu_e(\mu_2)$-weighted average of $ \mu(p,e) $ and that of $ \mu(q,e) $ 
depends solely on the ratio between the volume of $ F_{(1,e)}(p)$ and that of $ F_{(1,e)}(q)$.

For the sake of completeness, we also give the following corollary, which generalizes a similar result in \cite{ACL2} and can be viewed as a generalization of the quantum ergodicity over the leaves of the foliation generated by the integrable system. 

\begin{Corollary}
\label{corollary3}
Under Assumption (A), {\color{black} when $d \geq 3$,} if the joint Hamiltonian flow given by $X_{f_j}$'s is ergodic on $ F_{(1, e)} $ with respect to the Liouville measure for each $e \in \mathcal{F}$, 
then there is a family of distributions $\{ \Phi_{e}  \}_{e \in   \mathcal{F}  } \in \mathcal{D}' ( \partial D \times \partial D)$ as the Schwartz kernels of $\mathcal{K} _{e}$ such that they form a partition of the identity operator $Id$ as follows:
\begin{equation}\label{eq:pu1}
Id =    \int_{\mathcal{F} }  \mathcal{K}_{e} \,  de (e) \, ,
\end{equation}
which holds in the weak operator topology satisfying that for any given compact convex polytope $\mathcal{C} \subset (0,\infty) \subset \mathbb{R}^{k}$, there exists
$S(h) \subset J(h) := \{i \in \mathbb{N} : (\lambda^i_1(h), \cdots, \lambda^i_k(h) )\in \mathcal{C}  \}$ such that for all $\varphi \in C^\infty ( \partial D)$ and as $h\rightarrow +0$,
{\small
\begin{equation}\label{eq:pu2_a}
\max_{i \in S(h)}  \left|   \int_{ \partial D}  \varphi(x) \left( c_i \, | |D|^{-\frac{1}{2}} \phi^i (x) |^2 -
\int_{\mathcal{F} }   \frac{ \sigma_{x,F_{(1,e)}} \left( F_{(1,e)}(x)  \right)  }{  \sigma_{ F_{(1,e)}}  \left( F_{(1,e)} \right)   }   \, g_{i} (e)  \, de (e)  \right) d \sigma(x)
 \right| = o_{\mathcal{C}}(1)  \,.
\end{equation}
}In \eqref{eq:pu1},
\begin{equation}\label{eq:pu3_a}
 g_{i} (e)  :=  c_i 
 \langle  \mathcal{K}_{e}   |D|^{-\frac{1}{2}} \phi^i   , |D|^{-\frac{1}{2}} \phi^i  \rangle_{L^2(\partial D , d \sigma)}
\,,    \int_{\mathcal{F} }   g_{i} (e) \,   de (e) = 1 \, , \, 
\frac{\sum_{i \in S(h) } 1 }{ \sum_{i \in J(h) } 1 } = 1 + o_{\mathcal{C}}(1) \,.
\end{equation}
\end{Corollary}
\begin{proof}

The conclusion follows by noting that if the joint flow of $X_{f_j}$'s is ergodic with respect to $\sigma_{F_{(1,e)}}$ for each $e \in \mathcal{F}$, then $\sigma_{F_{(1,e)}} \in M_{X_{F},\text{erg}} ( F_{(1,e)})$ and we can take $ \nu = \delta_{\sigma_{F_{(1,e)}} } $ which is the Dirac measure of $\sigma_{F_{(1,e)}} \in M_{X_{F},\text{erg}} ( F_{(1,e)})$.
\end{proof}

By Corollary~\ref{corollary3}, we see that if the joint flow of $X_{f_j}$ is ergodic on $ F_{(1, e)} $ with respect to the Liouville measure for all $e \in \mathcal{F}$, then most of the function $ c_i ||D|^{-\frac{1}{2}} \phi^i (x) |^2$ weakly converges to a $\frac{g_{i} (e)}{\sigma_{F_{(1,e)}} (F_{(1,e)}) } $-weighted average of the volumns of $F_{(1,e)}(x)$ over $e \in \mathcal{F}$.  Therefore the value of eigenfunction at $x \in \partial D$ goes high as the volumn of $F_{(1,e)}(x)$ goes up for each leaf indexed by $e \in \mathcal{F}$.

\subsection{Localization/concentration of plasmon resonance at high-curvature points}

From Theorems~\ref{theorem1} and \ref{theorem2} in the previous subsection, it is clear that the relative magnitude of the NP eigenfunction $\phi^i $ at a point $x$  depends on the (weighted) volume of each leaf $F_{(1,e)}(x)$.   Therefore, in order to understand the localization of plasmon resonance, it is essential to obtain a better description of this volume.  Again, as we notice similar to \cite{ACL2}, this volume heavily depends on the magnitude of the second fundamental form $\mathscr{A}(x)$ at the point $x$.   As we will see in this subsection, in general, the higher the magnitude of the second fundamental form $\mathscr{A}(x)$ is, the larger the volume of the characteristic variety becomes. In particular, in a relatively simple case when the second fundamental forms at two points are constant multiple of each other, we have the following volume comparison.
\begin{Lemma}
\label{theorem3}
Let $p,q \in \partial D$ be such that $\mathscr{A}(p) = \beta \mathscr{A}(q) $ for some $\beta > 0$ and $g(p) = g(q)$. Then $|F_{(1,e)} (p)  | = \beta^{d-2} | |F_{(1,e)} (q)  | $.  We also have $$ \int_{ F_{(1,e)} (p)  }  | \xi |_{g(p)}^{1+ 2 \alpha }  d \sigma_{p,F_{(1,e)}} = \beta^{d- 1+ 2 \alpha } \int_{ F_{(1,e)} (q)  }  | \xi |_{g(q)}^{1+ 2 \alpha }  d \sigma_{q,F_{(1,e)}} . $$
\end{Lemma}
\begin{proof}
From $-2$ homogeneity of $H$, we have $H(p,\xi  ) = H(q,\xi / \beta )$, and therefore 
$\{F(x, \xi) = (\rho(1), e_2, \cdots, e_k)\}  = \beta \{ F(q,\xi)= (\rho(1), e_2, \cdots, e_k) \}$, which readily yields the conclusion of the theorem. 
\end{proof}

Similar to \cite{ACL2}, localization can be better understood via a more delicate volume comparison of the characteristic variety at different points with the help of Theorems \ref{theorem1} and \ref{theorem2} and Corollary \ref{corollary3}.  However, it is difficult to give a more explicit comparison of the volumes between $ F_{(1,e)} (p)  $ and $F_{(1,e)} (q) $ by their respective second fundamental forms $\mathscr{A}(p)$ and $\mathscr{A}(q)$. The following lemma provides a detour to control how the (weighted) volume of $F_{(1,e)} (p)  $ depends on the principal curvatures $\{ \kappa_i(p) \}_{i=1}^{d-1}$.


\begin{Lemma}
\label{theorem_volume}
Let $G_{\alpha}^{e} : \partial D \times  \mathbb{R}^{d-1} \rightarrow \mathbb{R}$ be given as
\begin{equation}\label{eq:d1}
G^e_{\alpha} \left( p,  \{ \kappa_i \}_{i=1}^{d-1} \right)  := \int_{\bigcap_{l=2}^k \left\{ f_l \left (p, \omega \sum_{i=1}^{d-1} \widetilde{\kappa_i  } \, \omega_i^2 \right) = e_l \right\} }   \left| \sum_{i=1}^{d-1} \widetilde{\kappa_i } \omega_i^2 \right|^{d-1 + 2 \alpha} \sqrt{ \sum_{i=1}^{d-1} \widetilde{\kappa_i }^2 \omega_i^2 }\  d \omega_{p,e}  ,
\end{equation}
where 
\begin{equation}\label{eq:d2}
 \widetilde{\kappa_i } := \sum_{j=1}^{d-1} \kappa_j  - \kappa_i . 
\end{equation}
Then the following inequality holds:
\begin{equation}\label{eq:d3}
G^e_{\alpha} \left( p,  \{ \kappa_i(p) \}_{i=1}^{d-1} \right)  \leq \int_{ F_{(1,e)} (p)  } | \xi |_{g(p)}^{1+2\alpha} \, d \sigma_{p,F_{(1,e)}} \leq 2 G^e_{\alpha} \left( p, \{ \kappa_i(p) \}_{i=1}^{d-1} \right).
\end{equation}
\end{Lemma}
\begin{proof}
We first fix a point $p$ and choose a geodesic normal coordinate with the principal curvatures along the directions $\xi_i$. In doing so, we can simplify the expression of $H(p,\xi) =1$. In fact, we then have
$$H(p,\xi) = \left(   \sum_{i=1}^{d-1} \widetilde{\kappa_i (p) } \, \xi_i^2 \right)^2 \bigg/  \left(   \sum_{i=1}^{d-1} \xi_i^2 \right)^3 \,. $$
Due to the $-2$ homogeneity of $H(p,\xi)$ with respect to $\xi$, we parametrize the surface $ \{ H(p,\xi) =1 \} $ by $\omega \in \mathbb{S}^{d-2}$ with $\xi (\omega) := r(\omega) \, \omega$.  Then we have
$$
r(\omega) =  \sum_{i=1}^{d-1} \widetilde{\kappa_i (p) } \, \omega_i^2  \,,
$$
Now with the above parametrization $\xi (\omega) = r(\omega) \, \omega$, we substitute to have
\beqnx
F_{(1,e)} (p)    =  \bigcap_{l=2}^k \left\{ \left(p, r(\omega) \omega \right)  \, :\,  \omega \in \mathbb{S}^{d-2},  f_l \left (p, \omega \sum_{i=1}^{d-1} \widetilde{\kappa_i (p) } \, \omega_i^2 \right) = e_l \right\} \,.
\eqnx
Writing the $(d-1-k)$-Hausdorff measure of the variety $ \bigcap_{l=2}^k \{ f_l \left (p, \omega \sum_{i=1}^{d-1} \widetilde{\kappa_i (p) } \, \omega_i^2 \right) = e_l  \} $ on $\mathbb{S}^{d-2}$ as:
\[
d \omega_{p,e} := \delta_{  \bigcap_{l=2}^k \left\{ \omega \in \mathbb{S}^{d-2} \, :\,  f_l \left (p, \omega \sum_{i=1}^{d-1} \widetilde{\kappa_i (p) } \, \omega_i^2 \right) = e_l \right\} } (d \omega) 
\]
and by following a similar argument to that of Lemma 4.5 in \cite{ACL2}, we can show that
\beqnx
 &&\left| \sum_{i=1}^{d-1} \widetilde{\kappa_i(p) } \omega_i^2 \right|^{d-1 + 2 \alpha} \sqrt{ \sum_{i=1}^{d-1} \widetilde{\kappa_i (p)}^2 \omega_i^2 } \, d \omega_{p,e}\\
 && \leq  | \xi |_{g(p)}^{1+2\alpha}  d \sigma_{p,F_{(1,e)} } \leq  2 \left| \sum_{i=1}^{d-1} \widetilde{\kappa_i(p) } \omega_i^2 \right|^{d-1 + 2 \alpha} \sqrt{ \sum_{i=1}^{d-1} \widetilde{\kappa_i (p)}^2 \omega_i^2 } \, \, d \omega_{p,e}  \,.
\eqnx

The proof is complete. 
\end{proof}

By Lemma~\ref{theorem_volume}, we can readily see that in order to compare the ration of the magnitudes of the NP eigenfunctions, one can actually compare the ration of the magnitudes of the principal curvatures at the respective point. For instance, if it happens that $\min_i |\widetilde{\kappa_i (p)}| \gg \max_i |\widetilde{\kappa_i (q)}| $, then it is clear that the weighted volume of  $F_{(1,e)} (p)$  is much bigger than that at $q$.

In the next subsection, we discuss a motivating example which shows how the above lemma can be simplified and provide precise and concrete description.

\subsection{A motivating example: surface with rotational symmetry} \label{example_revolution}  In what follow, we discuss a motivating example, which illustrates how the knowledge of another communting Hamiltonian simplifies the understanding of the Hamiltonian flow of our concern and provide explicit description of eigenfunction concentration.

\begin{Example}
   Consider that $D \subset \mathbb{R}^3$ is convex and take $ G:= \left \{ \begin{pmatrix} U & 0 \\ 0 & 1 \end{pmatrix} : U \in SO(2) \right \}  \subset SO(3) $ without loss of generality such that $G (D) = D$, i.e. $D$ (and hence $\partial D$) is invariant under the rotation group $ G $. 
Then, writing  $(x_1, x_2, x_3, \xi_1, \xi_2, \xi_3 )$ as a coordinate in $T^*( \mathbb{R}^3 )$, we recall the Lie algebra isomorophism:
\beqnx
j: \mathfrak{so}(3) = \{ A \in \mathbb{R}^{3\times 3} \,:\, A + A^T = 0 \} & \rightarrow & \mathbb{R}^3, \\
\begin{pmatrix}
0      & -a_3  &  a_2 \\
a_3  &  0      & -a_1\\
-a_2 &  a_1  & 0
\end{pmatrix}
&\mapsto&  (a_1, a_2, a_3),
\eqnx
where $j ([A,B]) = j (A) \times j (B)$ is the three dimensional cross product,
and the moment map:
\beqnx
\mu: T^*( \mathbb{R}^3 ) &\rightarrow & \mathfrak{so}(3)^* ,\\
\mu (x,\xi) & = &  \xi \times x \,.
\eqnx
Therefore the Hamiltonian that generates the one parameter subgroup $G \subset SO(3)$ is given by
\[
  \langle  \mu (x,\xi) , \overrightarrow{e_3}  \rangle = \langle  \xi \times x, (0,0,1)  \rangle = \xi_1 x_2 - \xi_2 x_1  \,.
\]
Now we define $F = (f_1, f_2)$, where $f_1(x,\xi) = \tilde{H} (x,\xi)$ defined as above and 
\beqnx
f_2 : T^*(\partial D) & \rightarrow  &\mathbb{R} \\
f_2(x,\xi) &=&     \langle  \mu ( \iota(x , \xi) ) ,  \overrightarrow{e_3}   \rangle,
\eqnx
where $\iota: T^*(\partial D) \rightarrow T^*( \mathbb{R}^3 )$ is the canonical embedding.
Notice that $\{f_i,f_j\} =  \delta_{ij}$ for $i =1,2$ since $G (\partial D) = \partial D $ and for all $g\in G$, the rotational symmetry gives that $g^* \tilde{H} = \tilde{H}$, and hence $\{f_2,f_1\} = X_{f_2} (\tilde{H}) = 0$. Hence $(f_1,f_2)$ forms a completely integrable system.

Next we gaze at $G^e_\alpha$, where $e = e_2 \in \mathcal{F} \subset \mathbb{R}$.
Let us first look into the case when $e_2 \neq 0$. With this system, we realize the two principal curvatures satisfy, for $i = 1,2$,
\[
\kappa_i (p) =  \kappa_i (  p_z ) \,,
\]
where $p \in \partial D \mapsto (p_x,p_y, p_z) \in \mathbb{R}^3$ is the canonical embedding (which can be represented via the parametrization $\mathbb{X}$), and hence in \eqref{eq:d2}:
\[
 \widetilde{\kappa_1 }(p) := \kappa_2 (p_z ) \,, \quad  \widetilde{\kappa_2 }(p) := \kappa_1 (p_z )\,.
\]
Moreover,  for $\omega = (\cos (\theta), \sin(\theta) ) \in \mathbb{S}^1$, writing $\{ v_i (p) \}_{i=1,2} \in T_p(\partial D) \cong T^*_p(\partial D)$ to be the principal directions, $A(p) v_i(p) = \kappa_i(p) v_i(p) $ under a geodesic normal coordinate and denoting $ \iota(p,\xi) = \left( p_x,p_y, p_z,  ( \xi_p )_x, ( \xi_p )_y, ( \xi_p )_z \right) \in \mathbb{R}^6$ with the choice that $(v_1(p) )_z = 0 $ (which makes the choice unique), we have $ f_2\left (p, \omega \sum_{i=1}^{2} \widetilde{\kappa_i (p) } \, \omega_i^2 \right) = e_2 $  if and only if $\theta$ satisfies
\beqn
 e_2 &=&  \left( \kappa_2^2 (p_z ) \cos^2 (\theta) + \kappa_1^2 (p_z ) \sin^2(\theta) \right)\notag\\
 &&\qquad\times  \left( 
  \cos(\theta) \left(   (v_1(p) )_x   p_y -  (v_1(p) )_y p_x    \right)   +
   \sin(\theta)  \left(   (v_2(p) )_x  p_y  - (v_2(p) )_y  p_x  \right)  \right) \notag \\
 & = & \left( \kappa_2^2 (p_z ) \cos^2 (\theta) + \kappa_1^2 (p_z ) \sin^2(\theta) \right) \notag\\
  &&\qquad\times  {   \sin\left(\theta + \tilde{ \theta }(p)  \right) } {\color{black}\times {
\sqrt{ \left(  (v_1(p) )_x   p_y -  (v_2(p) )_y p_x   \right)^2 + \left(  (v_2(p) )_x   p_y -  (v_1(p) )_y p_x   \right)^2 }
}},\notag \\
\label{eqn_leaf}
\eqn
where 
\[
\tilde{ \theta }(p) = \tan^{-1} \left( \frac{  (v_1(p) )_x   p_y -  (v_1(p) )_y p_x    }{ {\color{black}{ (v_2(p) )_x   p_y -  (v_2(p) )_y p_x  }} } \right).
\]
The RHS term in \eqref{eqn_leaf} is invariant by the rotational action (recall $g^* f_2 = f_2$ for all $g \in G$), and therefore is a function only of $(z_p,\theta)$.
$\theta$ can now be found via the tangent-half-angle formula and a solution to a sixth order polynomial. In any case, it is either an empty set if $e_2$ is too large, or a set of a finite number of points. Let us also denote $|e_{2}|_{\max} $ as the extremal value of $|e_{2}|$ such that the solution to \eqref{eqn_leaf} is non-empty, i.e. $$\mathcal{F} = [- |e_{2}|_{\max},  |e_{2}|_{\max} ] \,.$$ For a given $(p_z,e_2)$, we denote the number of solutions to \eqref{eqn_leaf} as $N(p_z,e_2)$.
If $ 0 < |e_2| \leq |e_{2}|_{\max}$, we can see that $1 \leq N(p_z,e_2) \leq 2$ and the solutions have the same value of $\cos(\theta)$.  Denoting them as $\theta_l(p_z,e_2)$, $l = 1,...,N(p_z,e_2) $, we have in \eqref{eq:d1}:
\[
\begin{split}
 G_{\alpha} \left( p, \{ \kappa_i(p) \}_{i=1}^{2} \right)  =& 2 \left(  \kappa_2 (p_z ) \cos^2 (\theta_1 (p_z,e_2)) + \kappa_1 (p_z ) \sin^2(\theta_1 (p_z,e_2)) \right)^{2 + 2 \alpha}\\
&\qquad\times \sqrt{     \kappa_2^2 (z_p ) \cos^2 (\theta_1 (p_z,e_2)) + \kappa_1^2 (z_p ) \sin^2(\theta_1 (p_z,e_2))    }\,.
\end{split}
\] 
It is now clear that $G^e_{\alpha} \left( p, \{ \kappa_i(p) \}_{i=1}^{2} \right)  $ increases seperately when either one of $ \kappa_1(p), \kappa_2(p) $ increases.
We remark that, in the case when $|e_2| = |e_{2}|_{\max}$, the only solution to \eqref{eqn_leaf} is when $\theta = 0$ if $e_2 > 0$ and $\theta = \pi$ if $e_2 < 0$, i.e. $N(p_z,e_2) = 1$.
We next gaze at when $e_2 = 0$.  When $(p_x,p_y) \neq (0,0) $, and taking $(v_1(p) )_z = 0 $, we only have $\theta = \frac{\pi}{2}$ and $\frac{\pi}{3}$ that satisfies \eqref{eqn_leaf}, and hence
\[
G^e_{\alpha} \left( p, \{ \kappa_i(p) \}_{i=1}^{2} \right)  = 2  \kappa_1^{3 + 2 \alpha} (p_z ) \,.
\]
When $p_x = p_y = 0$, we have $ \widetilde{\kappa_1 }(p)  =  \widetilde{\kappa_2 }(p) = 0$, and however we choose the directions $v_i(p)$, we then have that \eqref{eqn_leaf} is satisfied for all $ \theta \in [0, 2\pi)$, and in such a case:
\[
G^e_{\alpha} \left( p, \{ \kappa_i(p) \}_{i=1}^{2} \right)  =   \int_0^{2\pi}  \kappa (p_z )^{3 + 2 \alpha} d \theta  =  2 \pi \kappa (p_z )^{3 + 2 \alpha}  \,.
\]
Again, we can see that in either situations $G_{\alpha} \left( p, \{ \kappa_i(p) \}_{i=1}^{2} \right)  $ increases separately when either one of $ \kappa_1(p) $ increases.

Next, we look more closely into the flow induced by $X_{\tilde{\tilde{H}}}$ on $F_{(1, e_2)} $.  We first focus on the case when $ 0 < |e_2| < |e_{2}|_{\max} $. 
We denote $ p_{z,\max, e_2} $ and  $ p_{z,\min, e_2} $ as respectively the largest and smallest values of $p_z$ for $(p,\xi) \in F_{ (1, e_2) } $. Then we realize that
\[
\begin{split}
F_{ (1, e_2) } = & \bigcup_{p_z \in [p_{z,\min, e_2}, p_{z,\max, e_2}] } \bigg\{  \bigg( p, r(p_z,e_2) \left( \cos( \theta_l (p_z,e_2))  v_1 (p) + \sin( \theta_l (p_z,e_2))   v_2 (p) \right) \bigg) \\
&\hspace*{6cm}  l = 1,..., N(p_z,e_2)  \bigg\},
\end{split}
\]
where $ r(p_z,e_2) := \kappa_2^2 (z_p ) \cos^2 ( \theta_1 (p_z,e_2)) + \kappa_1^2 (z_p ) \sin^2( \theta_1 (p_z,e_2))  $.
It is remarked that when $p =  p_{z,\min, e_2} $ or $p = p_{z,\min, e_2} $, $\theta = 0 $ or $\pi$ and $p$ satisfies
\beqnx
 |e_2| =  \kappa_2^2 (p_z )  \sqrt{ p_x^2 +  p_y^2  },
\eqnx
where the RHS term is again a function only of $p_z$.
One can show that on $F_{(1, e_2)}$, we have $X_{\tilde{H}} = 0$ if and only if $\xi = 0$. 
Assume that we start with $(x,\xi) \in F_{(1, e_2)}$ and flow accordingly with $a_t = (p(t),\xi(t))$. Then there exists $[s_1,s_2]$ with the smallest interval length containing $0$ such that $p_z'(s_1) = p_z'(s_2) = 0 $ (which coincides with that when $p_z = p_{z,\max, e_2} $ or $p_z = p_{z,\min, e_2} $ ). Via reflection symmetry , we have for all $t \in \mathbb{R}, n \in \mathbb{Z}$:
\[
a_{t - s_1 + 2n (s_2-s_1) } = a_{ -(t - s_1) + n (s_2-s_1) } = a_{t - s_1 }.
\]
That is, $a_t$ is $2 (s_2-s_1) $ periodic and symmetric about $s_2$.  Hence the orbit of the flow generated by $X_{\tilde{H}} = 0$ on  $F_{ (1, e_2) }$ can either be periodic or quasi-periodic. It is realized that when $p_{z,\min, e_2}< p_z <  p_{z,\max, e_2}$, we have $N(p_z,e_2)= 2$; whereas, at $p_z =  p_{z,\max, e_2} $ or $p_z =  p_{z,\min, e_2}$, $\theta_l (p_z,e_2)) = 0 $ or $\pi$, and hence $N(p_z,e_2) = 1$.  
Therefore, tracing the points, we have
\[
F_{ (1, e_2) }  \cong \mathbb{T}^2,
\]
which is a smooth $2$-dimensional surface.
With that, there are only two cases if $ 0 < |e_2|  < |e_{2,\max}|$:
\begin{enumerate}
\item
When a Hamiltonian curve is periodic on $ \mathbb{T}^2$, a shift of the curve gives all the periodic orbits of the Hamiltonian flow. Hence we can exhaust the ergodic measures of $X_{\tilde{H}}$ on $F_{ (1, e_2) }$ indexed as $ \{ \mu_{(1,e_2),a} \}_{a \in [0,1)} $
where $\mu_{(1,e_2),a}$ contains the point $(p,\xi)$ such that $(p_x,p_y,p_z) = \left( \sqrt{ p_x^2+ p_y^2} \cos\left(2a \theta_{\text{per}}  \right),  \sqrt{ p_x^2+ p_y^2} \sin\left(2a \theta_{\text{per}} \right), p_{z,\max, e_2} \right)$ with $\theta_{\text{per}} = \cos^{-1} \left( \frac{ p_x(s_2) p_x(s_1) + p_y(s_2) p_y(s_1)}{ \sqrt{p_x^2(s_2)  + p_y^2(s_2)  } \sqrt{p_x^2(s_1)  + p_y^2(s_1)  }  } \right)$;
\item
When a curve is on the other hand quasi-periodic on $ \mathbb{T}^2$, we have that the flow of $X_{\tilde{H}}$ is ergodic on the torus $ \mathbb{T}^2$ on the Liouville measure restricted on $ \mathbb{T}^2$.
\end{enumerate}
If $|e_2| = |e_{2,\max}|$, we realize $N(p_z,e_2) = 1$, and 
\[
F_{ (1, e_2) }  \cong \mathbb{S}^1.
\]
If $e_2 = 0$, we can verify that $N(p_z,e_2) = 2$, and the solutions to \eqref{eqn_leaf} are $\theta = \frac{\pi}{2}$ and $\frac{\pi}{3}$. The flow is periodic with $\theta_{\text{per}}  = \cos^{-1} \left( \frac{ p_x(s_2) p_x(s_1) + p_y(s_2) p_y(s_1)}{ \sqrt{p_x^2(s_2)  + p_y^2(s_2)  } \sqrt{p_x^2(s_1)  + p_y^2(s_1)  }  } \right) = \pi $. We have
\[
F_{ (1, e_2) }  \cong \mathbb{T}^2 \,,
\]
and we can exhaust the ergodic measures of $X_{\tilde{H}}$ on $F_{ (1, e_2) }$ indexed as $ \mu_{(1,e_2),[0,1]} $
$\mu_{(1,e_2),a}$ containing the point $(p,\xi)$ such that $(p_x,p_y,p_z) = \left( \sqrt{ p_x^2+ p_y^2} \cos\left(2\pi a \right),  \sqrt{ p_x^2+ p_y^2} \sin\left(2\pi a \right), p_{z,\max, e_2} \right)$.

Finally, we notice that in all the cases, the joint flow given by $X_F$ is ergodic on $F_{ (1, e_2) }$ with respect to the Liouville measure for all possible values of $(1,e_2)$ such that $F_{ (1, e_2) }$ is non-empty:
it trivially holds when $F_{ (1, e_2) }  \cong \mathbb{S}^1$ and if the flow by $X_{\tilde{H}}$ is quasi-periodic on $F_{ (1, e_2) }  \cong \mathbb{T}^2$; and it holds when the flow by $X_{\tilde{H}}$ is periodic on $F_{ (1, e_2) }  \cong \mathbb{T}^2$ since $G$ induces a transitive on the index set $\{ \mu_{(1,e_2),a} \}_{a \in [0,1)} $.
Therefore Corollary \ref{corollary3} applies to obtain a density-one subsequence $i \in S(h) \subset J(h)$ of $ c_i | |D|^{-\frac{1}{2}} \phi^i (x) |^2 $ weakly converging to $\int \limits_{-|e_{2}|_{\max} }^{|e_{2}|_{\max}}  \frac{ \sigma_{x,F_{(1,e)}} \left( F_{(1,e_2)}(x)  \right)  }{  \sigma_{ F_{(1,e_2)}}  \left( F_{(1,e_2)} \right)   }   \, g_{i} (e_2)  \, de (e_2) $ where
\[
G^e_{-\frac{1}{2}} \left( p,  \{ \kappa_i(p) \}_{i=1}^{2} \right) \leq \sigma_{x,F_{(1,e)}} \left( F_{(1,e_2)}(x)  \right)  \leq 2 G^e_{-\frac{1}{2}} \left( p,  \{ \kappa_i(p) \}_{i=1}^{2} \right),
\]
with $G^e_{-\frac{1}{2}} \left( p,  \{ \kappa_i(\cdot) \}_{i=1}^{d-1} \right) $ in Lemma \ref{theorem_volume} as
\beqnx
G^e_{\frac{1}{2}} \left( p, \{ \kappa_i(p) \}_{i=1}^{2} \right) =
\begin{cases}
2 \sqrt{     \kappa_2^2 (z_p ) \cos^2 (\theta_1 (p_z,e_2)) + \kappa_1^2 (z_p ) \sin^2(\theta_1 (p_z,e_2))    }  & \text{ when }  e_2 \neq 0 \,, \\
2 \pi \kappa (p_z )  &  \text{ when }  e_2 = 0 \,. 
\end{cases}
\eqnx
It is clear now that the magnitude is monotonically increasing separately as each of the $\kappa_i$ increasing.

As a remark, a direct check of the (singular) fibration $\tilde{\pi}: \{\tilde{H} = 1\} \rightarrow \mathcal{F} := [ -|e_{2}|_{\max} ,  |e_{2}|_{\max}  ] $
with 
\[
\tilde{\pi}^{-1}(e_2) = 
\begin{cases}
\mathbb{T}^2 & \text{ when }  |e_2| < |e_{2,\max}| \,, \\
\mathbb{S}^1 & \text{ when } |e_2| = |e_{2,\max}| \, ,
\end{cases}
\]
is that, topologically,
\[
\{\tilde{H} = 1\} \cong S ( \mathbb{S}^2) \cong SO(3) \cong \mathbb{RP}^3
\]
via diffeomorphism where $S ( \mathbb{S}^2)$ is the sphere bundle of $\mathbb{S}^2$. Via the well known Gysin sequence, we obtain that
\[
H^0(\{\tilde{H} = 1\}; \mathbb{Z})  = \mathbb{Z}, \, H^1(\{\tilde{H} = 1\};\mathbb{Z} )  = 0, \, H^2(\{\tilde{H} = 1\};\mathbb{Z})  = \mathbb{Z}/2 \mathbb{Z},\, H^3(\{\tilde{H} = 1\};\mathbb{Z} )  = \mathbb{Z} \,.
\]
Now, one can verify that
\[
\bigcup_{e_2 \in   \mathcal{F}  } \tilde{\pi}^{-1}(e_2)  =  \mathbb{S}^1 \bigcup \left(  \mathbb{T}^2 \times (-|e_{2}|_{\max} ,  |e_{2}|_{\max} ) \right) \bigcup \mathbb{S}^1 \cong S ( \mathbb{S}^2) \,.
\]

    \end{Example}

\begin{Remark}
Our previous description and analysis may extend to systems that are nearly integrable via a perturbation analysis. Here, the system is given as a Komorogov non-degenerate perturbation of a (completely) integrable system (of class at least $C^{2(d-1), \alpha}$) via a classical KAM theory \cite{referee2_31,referee2_32,referee2_33}.  In such a case, if $k = n = d-1$, it is known that for an $\epsilon$-perturbation of the system, the flow will stay quasi-periodic on the surviving invariant Lagrangian tori which foliates/occupies $1-O(\sqrt \epsilon)$ of the space.   The invariant measures will then be localized on the these surviving invariant Lagrangian tori.
The dynamics in the remaining $O(\sqrt \epsilon)$-space may on the other hand be complicated, say e.g. Arnold diffusion may occur. However, when $d -1 = 2$, i.e. $d=3$, topological obstruction prevents the Arnold diffusion from happening. 

In the previous example with rotational symmetry with $d=3$, we may perturb a rotational symmetric shape to a shape of thin rod, and our result echoes with that in \cite{DLZ}.

\end{Remark}

\section{Localization/concentration of plasmon resonances for quasi-static wave scattering}\label{sect:5}

In this section, we extend all of the electrostatic results to the quasi-static case governed by the Helmholtz system. We refer to \cite{ACL2} for the discussion of the physical background, and moreover, by following the treatment therein, the concentration result in the quasi-static regime can be obtained by directly modifying the relevant results in the previous section. Hence, in what follows, we shall be brief in our discussion. 

Let $\varepsilon_0, \mu_0, \varepsilon_1 , \mu_1$ be real constants and assume that $\varepsilon_0$ and $\mu_0$ are positive. Let $D$ be given as that in Section~\ref{sect:1}, and set
\[
\mu_{D}  = \mu_1 \chi(D) +  \mu_0 \chi(\mathbb{R}^d \backslash \overline{D}),\quad 
\varepsilon_{D}  = \varepsilon_1 \chi(D) +  \varepsilon_0 \chi(\mathbb{R}^d \backslash \overline{D}). 
\]
Let $\omega\in\mathbb{R}_+$ denote the angular frequency of the operating wave. Set $k_0  := \omega \sqrt{ \varepsilon_0 \mu_0} $ and $k_1  := \omega \sqrt{ \varepsilon_1 \mu_1} $, with $\Im k_j\geq 0$, $j=0, 1$. 
Let $u_0$ be an entire solution to $(\Delta+k_0^2) u_0=0$ in $\mathbb{R}^d$. Consider the following Helmholtz scattering problem for $u\in H_{loc}^1(\mathbb{R}^d)$ satisfying
\beqn
    \begin{cases}
        \nabla \cdot (\frac{1}{\mu_{D} } \nabla u) + \omega^2 \varepsilon_{D} u = 0 &\text{ in }\; \mathbb{R}^d, \\[1.5mm]
         (\frac{\partial}{\partial |x| } - \mathrm{i} k_0 ) (u - u_0) = o(|x|^{ - \frac{d-1}{2}}) &\text{ as }\; |x| \rightarrow \infty. 
    \end{cases}
    \label{transmissionk}
\eqn
Henceforth, we assume that $\omega \ll 1$, or equivalently $k_0\ll 1$, which is known as the quasi-static regime. 

Let 
\beqn
    \Gamma_k (x-y) := C_{d} ( k|x-y| )^{- \frac{d-2}{2}} H^{(1)}_{\frac{d-2}{2}}(k|x-y|) ,
    \label{fundamental2}
\eqn
be the outgoing fundamental solution to $-(\Delta+k^2)$, where $C_d$ is some dimensional constant and  $H^{(1)}_{\frac{d-2}{2}}$ is the Hankel function of the first kind and order $(d-2)/2$. We introduce the following single-layer and NP operators associated with a given wavenumber $k\in\mathbb{R}_+$:
\beqn
    \mathcal{S}^k_{\partial D} [\phi] (x) &:=&  \int_{\partial D} \Gamma_k(x-y) \phi(y) d \sigma(y) ,\ \ x\in\partial D,\\
    {\mathcal{K}^k_{\partial D}}^* [\phi] (x) &:=& \int_{\partial D} \partial_{\nu_x} \Gamma_k(x-y) \phi(y) d \sigma(y) \, ,\ \ x\in\partial D. 
    \label{operatorK2}
\eqn

{\color{black}
Following the discussion in \cite{ACL2}, in the spirit of \eqref{eq:eigen1}, we consider the following generalized plasmon resonance problem when $\omega\in\mathbb{R}_+$: find a nontrivial $\phi \in H^{-1/2} ( \partial D, d \sigma)$ such that for some $m \in \mathbb{N} $,
\beqn
\bigg\{   \f{1}{2}  \left( \frac{1}{\mu_0} Id + \frac{1}{\mu_1} \, \left( \mathcal{S}^{k_1}_{\partial D} \right)^{-1}  \mathcal{S}^{k_0}_{\partial D} \right) +
\frac{1}{\mu_0} {\mathcal{K}^{k0}_{\partial D}}^* 
- \frac{1}{\mu_1} {\mathcal{K}^{k_1}_{\partial D}}^*   \left( \mathcal{S}^{k_1}_{\partial D} \right)^{-1}  \mathcal{S}^{k_0}_{\partial D}  \bigg\}^m  \phi  = 0.
\label{generalized_resonance}
\eqn
Here, $(\epsilon_1, \mu_1)$ is said to be a (pair of) generalized plasmonic eigenvalue.
We emphasize that $m$ must be finite in the equation.  When $m=1$, we refer to this problem as the plasmon resonance problem for $\omega\in\mathbb{R}_+$, and $(\epsilon_1, \mu_1)$ as a (pair of) generalized plasmonic eigenvalue.}

The following two lemmas in \cite{ACL2} characterize the plasmon resonance when $ \omega \ll  1$.

\begin{Lemma}
\label{close}
Under Assumption (A) and supposing $ \omega \ll  1$, we have that a solution $(  ( \mu_0,\mu_1,\eps_0,\eps_1,\omega ) , m, \phi_{\mu_0,\mu_1,\eps_0,\eps_1,\omega,m} ) $ to \eqref{generalized_resonance} with a unit $L^2$-norm possesses the following property for all $s \in \mathbb{R}$:
\beqnx
\begin{cases}
\| |D|^{s} \phi_{\mu_0,\mu_1,\eps_0,\eps_1,\omega, m}  - |D|^{s} \phi^i \|_{\mathcal{C}^{0}(\partial D)} &= \mathcal{O}_{i,s}(\omega^2),\medskip\\
 \lambda(\mu_0^{-1}, \mu_1^{-1}) - \tilde{\lambda}^i &= \mathcal{O}_{i}(\omega^2),\\
\end{cases}
\eqnx
for some eigenpair $( \tilde{\lambda}^i , \phi^i )$ of the Neumann-Poincar\'e operator $ {\mathcal{K}^*_{\partial D}}$, and $m \leq m_i$, where $m$ and $m_i$ signify the algebraic multiplicities of $\lambda$ and $\lambda_i$, respectively. 
\end{Lemma}

%

\begin{Lemma}
\label{existence}
Given any non-zero $\lambda_i \in \sigma( \mathcal{K}_{\partial D}^* )$ and for any $ (\tilde{\mu}_0, \tilde{\mu}_1) \in  D_i := \{ (\mu_0, \mu_1) \in \mathbb{C}^2\backslash \{(0,0)\} \,: \, \lambda(\tilde{\mu}_0^{-1}, \tilde{\mu}_1^{-1} ) = \tilde{\lambda}^i \,, \, \mu_0 - \mu_1 \neq 0 \,  \} $ (which is non-empty), there exists $ 0 < \omega_i \ll  1$ such that for all $\omega < \omega_i$, the set 
\beqnx
 &&\bigg\{ ( \mu_0,\mu_1,\eps_0,\eps_1 ) \in  \mathbb{C}^2 \backslash\{\mu_0 - \mu_1 = 0 \} \times ( \mathbb{C} \backslash \mathbb{R}^+ )^2 ; \\
&& \text{ there exists } m \in \mathbb{N}, \phi \in H^{-1/2} (\partial D, \sigma)  \text{ such that } (  ( \mu_0,\mu_1,\eps_0,\eps_1,\omega ) , \phi, m) \text { satisfies } \eqref{generalized_resonance} \bigg\}
\eqnx
forms a complex co-dimension $1$ surface in a neighborhood of $(\tilde{\mu}_0, \tilde{\mu}_1)$.
\end{Lemma}

By Lemma~\ref{existence}, we easily see that there are infinitely many choices of $(\varepsilon_1, \mu_1)$ such that the (genearalized) plasmon resonance occurs around $\tilde{\lambda}^i$. Combining with a similar perturbation argument as in the proof of Lemma~\ref{existence}, our conclusions of the plasmon resonance in the electrostatic case transfers to the Helmholtz transmission problem to show concentration of plasmon resonances at high-curvature points.  

\begin{Theorem}
\label{theorem4}
Given any $x \in \partial D$, let us consider $ \{ \chi_{x,\delta} \}_{\delta > 0}$ being a family of smooth nonnegative bump functions compactly supported in $B_{\delta} (x)$ with $ \int_{\partial D} \chi_{p,\delta} \, d \sigma = 1$.
Under Assumption (A),  {\color{black} when $d \geq 3$,} fixing a compact convex polytope $ \mathcal{C}  \subset \mathcal{F} \subset \mathbb{R}^{k-1}$, $[r,s] \subset{\mathbb{R}}$, $\alpha\in\mathbb{R}$ and $p, q \in \partial D$, 
we have a choice of $\delta(h)$ and $\omega(h) $ both depending on $\mathcal{C},p,q$ and $\alpha$ such that for any $\omega <  \omega (h)$, there exists 
$$\left( ( \mu_{0,i},\mu_{1,i},\eps_{0,i},\eps_{1,i}, \omega ), m_i , \phi_{ \mu_{0,i},\mu_{1,i},\eps_{0,i},\eps_{1,i}, \omega, m_i }  \right) $$ solving \eqref{generalized_resonance}, and as $h\rightarrow +0$, we have $\delta(h) \rightarrow 0$,  $\omega(h) \rightarrow 0$ and
\begin{equation}\label{concentration11}
\begin{split}
   & \frac{
        \sum_{\left\{\left( \lambda^i_1(h), \lambda^i_2(h), \cdots, \lambda^i_k(h)\right) \in [r,s] \times \mathcal{C}\right\}} c_i \,
        \int_{\partial D} \chi_{p,\delta (h) }(x) | |D|^{\alpha}  \phi_{ \mu_{0,i},\mu_{1,i},\eps_{0,i},\eps_{1,i}, \omega, m_i }  (x) |^2 d \sigma (x) 
    }
    {
        \sum_{\left\{\left(\lambda^i_1 (h), \lambda^i_2(h), \cdots, \lambda^i_k(h)\right)\in [r,s] \times \mathcal{C}\right\}} c_i \,
        \int_{\partial D} \chi_{q,\delta (h) }(x) | |D|^{\alpha}  \phi_{ \mu_{0,i},\mu_{1,i},\eps_{0,i},\eps_{1,i}, \omega, m_i }  (x) |^2 d \sigma (x) 
    } \\ &=   \frac{     \int_{\mathcal{F} } \int_{ F_{(1,e)}(p)  }  | \xi |_{g(y)}^{1+ 2 \alpha } d \sigma_{p,F_{(1,e)}} \, h(e) d e  }{   \int_{\mathcal{F} } \int_{F_{(1,e)}(q)  }  | \xi |_{g(y)}^{1+ 2 \alpha } d \sigma_{q,F_{(1,e)}} \, h(e) d e   } + o_{\mathcal{C},r,s,p,q,\alpha}(1), \notag \\
\end{split}
\end{equation}
where $c_i := | \phi^i |_{H^{-\frac{1}{2}}(\partial D , d \sigma)}^{-2} $.  Here, the little-$o$ depends on $\mathcal{C},r,s,p,q$ and $\alpha$.
\end{Theorem}
\begin{proof}
From Theorem \ref{theorem1},  we have a choice of $\delta(h)$ depending on $\mathcal{C},p,q$ and $\alpha$ such that,
for any given $\varepsilon>0$, there exists $h_0$ depending on $\mathcal{C},p,q,\alpha$ such that for all $h < h_0$,
\begin{equation}
\begin{split}
  &  \Bigg| \frac{ \sum_{\left\{\left(\lambda^i_1(h), \lambda^i_2(h), \cdots, \lambda^i_k(h)\right)\in [r,s] \times \mathcal{C}\right\}} c_i  \int_{\partial D} \chi_{p,\delta (h) }(x) | | D |^{\alpha}   \, \phi^i (x) |^2 d \sigma }{ \sum_{\left\{\left( \lambda^i_1(h), \lambda^i_2(h), \cdots, \lambda^i_k(h)\right)\in [r,s] \times \mathcal{C}\right\}} c_i  \int_{\partial D} \chi_{q,\delta (h) }(x) | | D |^{\alpha}   \, \phi^i (x) |^2 d \sigma } \\
  &\hspace*{2cm}-  \frac{     \int_{\mathcal{F} } \int_{ F_{(1,e)}(p)  }  | \xi |_{g(y)}^{1+ 2 \alpha } d \sigma_{p,F_{(1,e)}} \, h(e) d e  }{   \int_{\mathcal{F} } \int_{F_{(1,e)}(q)  }  | \xi |_{g(y)}^{1+ 2 \alpha } d \sigma_{q,F_{(1,e)}} \, h(e) d e   }  \Bigg| \\ &\leq \varepsilon.
    \label{ineq11a}
\end{split}
\end{equation}

Now for each $h < h_0$, from Lemma \ref{existence}, there exists $\tilde{\omega}(h) := \min_{\{ i \in \mathbb{N}:  \left( \lambda^i_1(h), \lambda^i_2(h), \cdots, \lambda^i_k(h)\right) \in \mathcal{C}\} }\{ \omega_i \}$ such that for all $\omega < \omega (h)$, there exists 
$$\left( ( \mu_{0,i},\mu_{1,i},\eps_{0,i},\eps_{1,i}, \omega ), m_i , \phi_{ \mu_{0,i},\mu_{1,i},\eps_{0,i},\eps_{1,i}, \omega, m_i }  \right) $$ solving \eqref{generalized_resonance}.   By Lemma \ref{close}, upon a rescaling of $\phi_{ \mu_{0,i},\mu_{1,i},\eps_{0,i},\eps_{1,i}, \omega, m_i } $ while still denoting it as $\phi_{ \mu_{0,i},\mu_{1,i},\eps_{0,i},\eps_{1,i}, \omega, m_i } $, we have
\beqnx
\| |D|^{\alpha} \phi_{ \mu_{0,i},\mu_{1,i},\eps_{0,i},\eps_{1,i}, \omega, m_i }  - |D|^{\alpha} \phi^i \|_{\mathcal{C}^{0 }(\partial D)} \leq  C_{i,\alpha} \omega^2.
\eqnx
In particular, we can make a smaller choice of  $\omega(h) < \tilde{\omega}(h) $ depending on $\mathcal{C},r,s,p,q,\alpha$ such that for all $\omega < \omega(h)$, we have
\beqnx
\| | |D|^{\alpha} \phi_{ \mu_{0,i},\mu_{1,i},\eps_{0,i},\eps_{1,i}, \omega, m_i } |^2  - | |D|^{\alpha} \phi^i |^2 \|_{\mathcal{C}^{0 }(\partial D)} \leq & 10^{-2} \varepsilon /\Theta,
\eqnx
where
\[
\Theta:= \sum_{\left( \lambda^i_1 (h), \lambda^i_2(h), \cdots, \lambda^i_k(h)\right) \in [r,s] \times \mathcal{C}} \hspace{-0.99cm}c_i / \min \left\{1,  \min_{y = p,q} \left\{  \left( \sum_{\left(\lambda^i_1(h), \lambda^i_2(h), \cdots, \lambda^i_k(h)\right) \in [r,s]  \times \mathcal{C}} \vartheta_i  \right)^{-2} \right \}  \right \} \,,
\]
and
\[
\vartheta_i:=c_i  \int_{\partial D} \chi_{y,\delta (h) }(x) | | D |^{\alpha}   \, \phi^i (x) |^2 d \sigma. 
\]

Therefore, with this choice of $\omega(h) $, we have, for all $\omega < \omega(h)$

\begin{equation}\label{ineq11b}
\begin{split}
&\Bigg|
\frac{
        \sum_{\left\{\left( \lambda^i_1(h), \lambda^i_2(h), \cdots, \lambda^i_k(h)\right) \in [r,s] \times \mathcal{C}\right\}} c_i \,
        \int_{\partial D} \chi_{p,\delta (h) }(x) | |D|^{\alpha}  \phi_{ \mu_{0,i},\mu_{1,i},\eps_{0,i},\eps_{1,i}, \omega, m_i }  (x) |^2 d \sigma (x) 
    }
    {
        \sum_{\left\{\left(\lambda^i_1 (h), \lambda^i_2(h), \cdots, \lambda^i_k(h)\right)\in [r,s] \times \mathcal{C}\right\}} c_i \,
        \int_{\partial D} \chi_{q,\delta (h) }(x) | |D|^{\alpha}  \phi_{ \mu_{0,i},\mu_{1,i},\eps_{0,i},\eps_{1,i}, \omega, m_i }  (x) |^2 d \sigma (x) 
    } \\
& \qquad\qquad\qquad-
  \frac{     \int_{\mathcal{F} } \int_{ F_{(1,e)}(p)  }  | \xi |_{g(y)}^{1+ 2 \alpha } d \sigma_{p,F_{(1,e)}} \, h(e) d e  }{   \int_{\mathcal{F} } \int_{F_{(1,e)}(q)  }  | \xi |_{g(y)}^{1+ 2 \alpha } d \sigma_{q,F_{(1,e)}} \, h(e) d e   } 
 \Bigg| \leq \varepsilon.
\end{split}
\end{equation}
Combining \eqref{ineq11b} with \eqref{ineq11a} readily yields our conclusion.

The proof is complete. 
\end{proof}

In a similar manner, we obtain the following result.
\begin{Theorem}
\label{theorem5}
Under Assumption (A), {\color{black} when $d \geq 3$,} given a compact convex polytope $\mathcal{C} \subset \mathbb{R} \times \mathcal{F} \subset \mathbb{R}^k$, there exists $S(h) \subset J(h) := \{i \in \mathbb{N} : \left(\lambda^i_1(h), \lambda^i_2(h), \cdots, \lambda^i_k(h)\right) \in \mathcal{C} \}$ and $\omega(h) $ such that, for all $\varphi \in C^\infty ( \partial D)$, we have for any $\omega < \omega (h)$, there exists 
$$\left( ( \mu_{0,i},\mu_{1,i},\eps_{0,i},\eps_{1,i}, \omega ), m_i , \phi_{ \mu_{0,i},\mu_{1,i},\eps_{0,i},\eps_{1,i}, \omega, m_i }  \right) $$ solving \eqref{generalized_resonance}, such that as $h\rightarrow+0$, we have $\omega(h) \rightarrow 0$ and
\begin{equation}\label{eq:tt1}
\begin{split}
& \max_{i \in S(h)}  \bigg|   \int_{ \partial D}  \varphi(x) \bigg( c_i \, | |D|^{-\frac{1}{2}} \phi_{ \mu_{0,i},\mu_{1,i},\eps_{0,i},\eps_{1,i}, \omega, m_i }  (x) |^2\\
& \hspace*{2cm} - \int_{\mathcal{F} } \int_{M_{X_{F},\text{erg}}( F_{(1,e)} )}   \mu(x,e)     \, g_{i} (\mu_e)  \,  d\nu_e \, (\mu_e  )  \, de (e)    \bigg) d \sigma(x)
 \bigg|= o_{r,s}(1)  \,.
 \end{split}
\end{equation}
Here, $S(h)$, $\{ g_{i} : \bigcup_{e \in \mathcal{F}} M_{X_{F},\text{erg}}( F_{(1,e)} )  \rightarrow \mathbb{C}  \}_{i \in \mathbb{N} }$ and $ \mu(p,e)$ are described as in Theorem~\ref{theorem2}.
In particular, we remind that
\beqnx
 \frac{\int_{M_{X_{F},\text{erg}}(F_{(1,e)})}  \mu(p,e)  \, d \nu_e(\mu_e)  }{  \int_{M_{X_{F},\text{erg}}(F_{(1,e)})}  \mu(q,e)  \, d \nu_e(\mu_e) }  =  \frac{ \int_{F_{(1,e)}(p) } d \sigma_{p,F_{(1,e)}}  }{  \int_{F_{(1,e)}(q) } d \sigma_{q,F_{(1,e)}}  } \text{ a.e. } (d \sigma \otimes d \sigma) (p,q)\,.
\eqnx
if the joint Hamiltonian flow given by $X_{f_j}$'s is ergodic on $ F_{(1, e)} $ with respect to the Louville measure for each $e \in \mathcal{F}$,  then 
\beqnx
&&\max_{i \in S(h)}  \Bigg|   \int_{ \partial D}  \varphi(x) \bigg( c_i \, | |D|^{-\frac{1}{2}} \phi_{ \mu_{0,i},\mu_{1,i},\eps_{0,i},\eps_{1,i}, \omega, m_i }  (x) |^2 \\
&& \qquad\qquad-  \int_{\mathcal{F} }   \frac{ \sigma_{x,F_{(1,e)}} \left( F_{(1,e)}(x)  \right)  }{  \sigma_{ F_{(1,e)}}  \left( F_{(1,e)} \right)   }   \, g_{i} (e)  \, de (e)   \bigg) d \sigma(x)
 \Bigg|= o_{\mathcal{C}}(1)  \,,
\eqnx
where $\{ g_{i} : \mathcal{F}  \rightarrow \mathbb{C}  \}_{i \in \mathbb{N} }$  is now defined as in Corollary~\ref{corollary3}.
\end{Theorem}

\begin{proof}
Let the compact convex polytope $\mathcal{C}$ be given.  Consider $\varphi \in \mathcal{C}^{\infty}(\partial D)$.
Given $\eps > 0$, by Theorem \ref{theorem2} and considering $h_0$ small enough such that for all $ h < h_0$, we have
\beqnx
\max_{i \in S(h)}  \left|   \int_{ \partial D}  \varphi(x) \left( c_i \, |D|^{-\frac{1}{2}} \phi^i (x) |^2  - \int_{\mathcal{F} } \int_{M_{X_{F},\text{erg}}( F_{(1,e)} )}   \mu(x,e)     \, g_{i} (\mu_e)  \,  d\nu_e \, (\mu_e  )  \, de (e)    \right) d \sigma(x)
 \right| \leq \eps .
\eqnx
Now, for each $h < h_0$, from Lemma \ref{existence}, there exists $\tilde{\omega}(h) = \min \left\{ \min_{i \in  S(h)} \omega_i, 1 \right\}$ such that for all $\omega < \tilde{ \omega }(h)$, there exists 
$$\left( ( \mu_{0,i},\mu_{1,i},\eps_{0,i},\eps_{1,i}, \omega ), m_i , \phi_{ \mu_{0,i},\mu_{1,i},\eps_{0,i},\eps_{1,i}, \omega, m_i }  \right) $$ solving \eqref{generalized_resonance}.   By Lemma \ref{close}, again upon a rescaling of $\phi_{ \mu_{0,i},\mu_{1,i},\eps_{0,i},\eps_{1,i}, \omega, m_i } $ while still denoting it as $\phi_{ \mu_{0,i},\mu_{1,i},\eps_{0,i},\eps_{1,i}, \omega, m_i } $, we have
\begin{equation}\label{eq:ff1}
\begin{split}
& \max_{i \in S(h)} c_i  \left|   \int_{ \partial D}  \varphi(x) \left(  \, | |D|^{-\frac{1}{2}} \phi^i (x) |^2 -  \, | |D|^{-\frac{1}{2}} \phi_{ \mu_{0,i},\mu_{1,i},\eps_{0,i},\eps_{1,i}, \omega, m_i }  (x) |^2  \right) d \sigma(x)
 \right|\\
 & \leq C_{S(h)} \| \varphi \|_{C^{0}(\partial D ) } \, \omega^2.
 \end{split}
\end{equation}
We may now choose 
$$ \omega(h) \leq \min \left\{ \eps,  \tilde{ \omega }(h),  \tilde{\omega}(h) / C_{S(h)} \right\} \, .$$
Then for all $\omega < \omega(h)$, we finally have from \eqref{eq:ff1} and Corollary \ref{corollary3} that
\[
\begin{split}
&  \max_{i \in S(h)}  \bigg|   \int_{ \partial D}  \varphi(x) \bigg( c_i \, | |D|^{-\frac{1}{2}} \phi_{ \mu_{0,i},\mu_{1,i},\eps_{0,i},\eps_{1,i}, \omega, m_i }  (x) |^2\\
& \hspace*{2cm}  - \int_{\mathcal{F} } \int_{M_{X_{F},\text{erg}}( F_{(1,e)} )}   \mu(x,e)     \, g_{i} (\mu_e)  \,  d\nu_e \, (\mu_e  )  \, de (e)     \bigg) d \sigma(x)
 \bigg| \\
 \leq & \left( 1+ \| \varphi \|_{C^{0}(\partial D ) } \right)  \eps  \, .
\end{split}
\]

The proof is complete. 
\end{proof}

We would like to remark that a similar conclusion holds for the explicit motivating example discussed in Section \ref{example_revolution} in the quasi-static case when $\omega\ll 1$, which we choose not to repeat.

\section*{Acknowledgment}
The work of H Liu was supported by Hong Kong RGC General Research Funds (project numbers, 11311122, 11300821 and 12301420), NSFC/RGC Joint Research Grant N\_CityU101/21 and ANR/RGC Joint Research Grant A\_CityU203/19.

\end{document}